\documentclass[reqno,11pt,final]{amsart}

\usepackage[latin1]{inputenc}

\usepackage{amsmath,amssymb,amsthm}
\usepackage{enumitem,epsfig,psfrag}
\setenumerate{leftmargin=*}

\usepackage[colorlinks,bookmarks,linkcolor=black,citecolor=black]{hyperref}
\usepackage[notref,notcite,color]{showkeys}

\let\subset\subseteq 
\let\eps\varepsilon
\let\rho\varrho
\def\dcup{\dot\cup}

\def\NATS{\mathbb{N}}
\def\Prob{\mathbb{P}}
\def\Exp{\mathbb{E}}
\def\bigO{\mathcal{O}}
\def\deltaH{\delta^*}

\def\cM{\mathcal{M}}
\def\cI{\mathcal{I}}
\def\cF{\mathcal{F}}
\def\cG{\mathcal{G}}
\def\cH{\mathcal{H}}
\def\cE{\mathcal{E}}
\def\cR{\mathcal{R}}

\def\cS{\mathcal{S}}

\let\l\ell

\def\clqed{\hfill\scalebox{.6}{$\Box$}}

\def\itmit#1{\rm ({\it #1\,})}
\def\rom{\itmit{\roman{*}}}
\def\abc{\itmit{\alph{*}}}

\newtheorem{theorem}{Theorem}
\newtheorem{lemma}[theorem] {Lemma}   
\newtheorem{corollary}[theorem] {Corollary}   
  
\newtheorem{proposition}[theorem] {Proposition}

\newtheorem{definition}[theorem] {Definition} 
\theoremstyle{remark} 
 
\newtheorem{AuxiliaryCl}[theorem]{Claim} 

\newcommand{\oldqed}{}
\newenvironment{clproof}[1][Proof]{
  \renewcommand{\oldqed}{\qedsymbol}
  \renewcommand{\qedsymbol}{\clqed}
  \begin{proof}[#1]
}{
  \end{proof}
  \renewcommand{\qedsymbol}{\oldqed}
}

\newcommand{\EMAIL}[1]{  \textit{E-mail}: \texttt{#1}} 

\newcommand{\By}[2]{\overset{\mbox{\tiny{#1}}}{#2}}
\newcommand{\ByRef}[2]{   \By{\eqref{#1}}{#2} }

\newcommand{\gBy}[1]{     \By{#1}{>} }

\newcommand{\geBy}[1]{    \By{#1}{\ge} }
\newcommand{\eqByRef}[1]{ \ByRef{#1}{=} }

\newcommand{\leByRef}[1]{ \ByRef{#1}{\le} }
\newcommand{\geByRef}[1]{ \ByRef{#1}{\ge} }

\newcommand{\Tur}[2]{\mathrm{T}_{#1}({#2})} 
\newcommand{\tur}[2]{t_{#1}({#2})}          
\newcommand{\turm}[3]{t_{#1}({#2,#3})}
\newcommand{\Turm}[3]{\mathrm{T}_{#1}({#2,#3})}
\newcommand{\neighbor}{\Gamma}

\newcommand{\Gnq}[1][q]{G(n,#1)}

\DeclareMathOperator{\link}{Link}
\DeclareMathOperator{\Bin}{Bin}

\title{Tur\'annical hypergraphs}

  \author[Peter Allen]{Peter Allen*}
  \author[Julia B\"ottcher]{Julia B\"ottcher*}

  \thanks{
    *
    Instituto de Matem\'atica e Estat\'{\i}stica, Universidade de
    S\~ao Paulo, Rua do Mat\~ao~1010, 05508--090~S\~ao Paulo, Brazil.
   \EMAIL{allen|julia@ime.usp.br}}

  \author[Jan Hladk\'y]{Jan Hladk\'y\dag}

  \thanks{
    \dag\ 
    DIMAP and Department of Computer Science,
    University of Warwick,
    Coventry, CV4~7AL, UK
    \EMAIL{honzahladky@gmail.com}}

  \author[Diana Piguet]{Diana Piguet\ddag}

\thanks{\ddag\ 
    School of Mathematics, 
    University of Birmingham, 
    Edgbaston, Birmingham, 
    B15~2TT,
    UK
    \EMAIL{D.Piguet@bham.ac.uk}}

  \thanks{
    PA, JH, and DP were supported by DIMAP, EPSRC award
    EP/D063191/1.
    PA was partially supported by FAPESP (Proc.~2010/09555-7), and
    JB by FAPESP (Proc.~2009/17831-7).
    PA and JB are grateful to NUMEC/USP, N\'ucleo de Modelagem Estoc\'astica e
    Complexidade of the University of S\~ao Paulo, for supporting this
    research.}

\date{\today}

\begin{document}
\begin{abstract}
  This paper is motivated by the question of how global and dense restriction
  sets in results from extremal combinatorics can be replaced by less global and
  sparser ones. The result we consider here as an example is Tur\'an's theorem,
  which deals with graphs $G=([n],E)$ such that no member of the
  \emph{restriction set} $\cR=\binom{[n]}{r}$ induces a copy of~$K_r$.

  Firstly, we examine what happens when this restriction set is replaced by
  $\cR=\{X\in \binom{[n]}{r}\colon X\cap[m]\neq\emptyset\}$.  That is, we
  determine the maximal number of edges in an $n$-vertex such that no $K_r$
  hits a given vertex set.
  
  Secondly, we consider sparse random restriction sets.  An $r$-uniform
  hypergraph~$\mathcal R$ on vertex set $[n]$ is called
  \emph{Tu\-r\'an\-ni\-cal} (respectively \emph{$\eps$-Tur\'an\-ni\-cal}), if
  for any graph $G$ on $[n]$ with more edges than the Tur\'an number
  $\tur{r}{n}$ (respectively $(1+\eps)\tur{r}{n}$\,), no hyperedge
  of~$\cR$ induces a copy of~$K_r$ in~$G$.  We determine the thresholds for
  random $r$-uniform hypergraphs to be Tur\'annical and to
  be $\eps$-Tur\'annical.
  
  Thirdly, we transfer this result to sparse random graphs, using techniques
  recently developed by Schacht [Extremal results for random discrete
  structures] to prove the Kohayakawa-{\L}uczak-R\"odl Conjecture on Tur\'an's
  theorem in random graphs.
\end{abstract}
\maketitle

\section{Introduction}
\label{sec:intro}

Tur\'an's theorem~\cite{Tur}, whose proof in 1941 marks the birth of
extremal graph theory, determines the maximal number of edges in an
$n$-vertex graph without cliques of size $r$. Let $\Tur{r}{n}$ denote the
complete \emph{balanced} $(r-1)$-partite graph on $n$ vertices (i.e., the part
sizes of $\Tur{r}{n}$ are as equal as possible) and $\tur{r}{n}$ the
number of its edges.

\begin{theorem}[Tur\'an~\cite{Tur}] 
  Given $n$ and $r$, let $G$ be an $n$-vertex graph that contains no copy of
  $K_r$. Then $G$ has at most $\tur{r}{n}$ edges.
\end{theorem}

Since 1941, many extensions of Tur\'an's theorem have been
established. Highlights certainly include the Erd\H{o}s-Stone
theorem~\cite{ErdosStone1946} which generalises the result from cliques to
arbitrary $r$-chromatic graphs, and the recent proofs by
Schacht~\cite{Schacht:KLR} and Conlon and Gowers~\cite{ConGow:KLR} of the
Kohayakawa-{\L}uczak-R\"odl conjecture on Tur\'an's theorem in random
graphs.

These extensions, however, do not deviate from the original result as far
as the following aspect is concerned. The restrictions they impose on the
class of objects under study are \emph{global} and \emph{dense}.
More concretely, they require for \emph{every} $k$-tuple of vertices that
these vertices do not host a copy of a given graph~$K$ on~$k$ vertices.
In this paper we are interested in the question of how weakening these
restrictions to less global or sparser ones (that is, forbidding $K$-copies
only for certain $k$-tuples but not all) can influence the conclusion of
the original Tur\'an theorem.

To make a first move, let us investigate the following natural question
which replaces the global restriction of Tur\'an's theorem by a non-global
one. \emph{How many edges can an $n$-vertex graph have such that no $K_r$
intersects a given set of $m$ vertices in this graph?} Our first result
states that the answer is 
\begin{equation}\label{eq:turm}
  \turm{r}{n}{m}:=\begin{cases}
    \tur{r}{n}
    & \text{if $n\le(r-1)m$}\,, \\
    \binom{n}{2} - nm + (r-1)\binom{m+1}{2}
   & \text{otherwise}\,.
  \end{cases}
\end{equation}

\begin{theorem}\label{thm:TuranExt} 
  Given $r\geq 3$ and $m\leq n$, let $G$ be any $n$-vertex graph and
  $M\subset V(G)$ contain $m$ vertices. If no copy of
  $K_r$ in $G$ intersects $M$, then $e(G)\le \turm{r}{n}{m}$.
  Moreover, if $n\le (r-1)m$ and $e(G)=\turm{r}{n}{m}$
  then $G$ is isomorphic to $\Tur{r}{n}$.
\end{theorem}

This means that for fixed~$n$, as~$m$ decreases
from~$n$ (the original scenario of Tur\'an's theorem) to~$0$ (no
restrictions at all) the extremal number~$\turm{r}{n}{m}$ stays equal to
$\tur{r}{n}$ until $m=n/(r-1)$ and then slowly increases (as a quadratic
function in~$m$) to~$\binom{n}{2}$.

A natural way of formalising this deviation from Tur\'an's theorem is to
introduce a hypergraph which contains a hyperedge for every restriction and
then ask for the maximal number~$k$ of edges in a graph respecting these
restrictions. The following definition makes this precise. We shall
distinguish between the case when~$k$ is still the Tur\'an number and when
it is bigger by a certain percentage.

\begin{definition}[Tur\'annical]
\label{def:Turannical}
  Let $r\ge3$ be an integer.
  Let $\mathcal{F}=(V,\mathcal E)$ be an $n$-vertex, $r$-uniform
  hypergraph with vertex set~$V$, which we also occasionally call
  \emph{restriction hypergraph}.
  The hypergraph~$\mathcal{F}$ \emph{detects}
  a graph~$G=(V,E)$ if some $F\in\mathcal{E}$ induces a copy
  of~$K_r$ in~$G$.
  We say that~$\mathcal{F}$ is \emph{exactly Tur\'an\-ni\-cal} or
  simply \emph{Tur\'annical}, if for all graphs $G=(V,E)$ with
  $e(G)>\tur{r}{n}$ the hypergraph $\cF$ detects~$G$.
  In addition, $\mathcal{F}$ is \emph{$\eps$-approximately
  Tur\'annical} or simply \emph{$\eps$-Tur\'annical} if for all graphs $G=(V,E)$
  with $e(G)>(1+\eps)\tur{r}{n}$ the hypergraph $\cF$ detects $G$.
\end{definition}

In other words, a restriction hypergraph is Tur\'annical if it detects all
graphs whose density is large enough that \emph{one} copy of~$K_r$ is forced to
exist, and it is approximately Tur\'annical if it detects all graphs whose
density forces a \emph{positive density} of copies of~$K_r$ to exist (cf.~the
so-called super-saturation theorem, Theorem~\ref{thm:TuranDensity}, by Erd\H{o}s
and Simonovits~\cite{ErdSim:Supersaturated}).

In this language Tur\'an's theorem states that the complete $r$-uniform hypergraph is
Tur\'an\-ni\-cal and Theorem~\ref{thm:TuranExt} concerns restriction
hypergraphs with all hyperedges meeting a specified set of vertices~$M$
(see also the reformulation in Theorem~\ref{thm:Intersection}).

Another natural question is whether the dense complete $r$-uniform
restriction hypergraph from Tur\'an's theorem may be replaced by a much
sparser one. Here, hypergraphs formed by \emph{random restrictions} might
appear promising candidates: A random $r$-uniform hypergraph
$\cR^{(r)}(n,p)$ with hyperedge probability $p$ is a hypergraph on vertex set $[n]$
where hyperedges from $\binom{[n]}{r}$ exist independently from each other with
probability $p$. And in fact, we will show that $\cR^{(r)}(n,p)$ for appropriate
values of $p=p_n$ produces the Tur\'annical hypergraphs and
$\eps$-Tur\'annical hypergraphs with the fewest number of hyperedges, up to
constant factors (compare Proposition~\ref{prop:density} with
Theorems~\ref{thm:random:approx} and~\ref{thm:random:exact}).  In addition,
building on the aforementioned work of Schacht~\cite{Schacht:KLR} we obtain
a corresponding result for the random graphs version of Tur\'an's
theorem (see Theorem~\ref{thm:TurForGnp}).

Before we state and explain these results in detail in the following
section, let us remark that the observed behaviour concerning the
\emph{evolution} of $\cR^{(r)}(n,p)$ as we decrease the density of the
random restrictions is somewhat different from the one described for
Theorem~\ref{thm:TuranExt} above: When~$p$ decreases from~$1$ to~$0$, then
$\cR^{(r)}(n,p)$ stays (asymptotically almost surely) Tur\'annical for a long time, until
$p_n\sim n^{3-r}$. Then, between $p_n\sim n^{3-r}$ and $p_n\sim n^{2-r}$
the hypergraph $\cR^{(r)}(n,p)$ is $\eps$-Tur\'annical for
\emph{arbitrarily} small (but fixed) $\eps>0$, and for even smaller $p_n$
the hypergraph $\cR^{(r)}(n,p)$ fails to be $\eps$-Tur\'annical for
\emph{any} non-trivial~$\eps$. As we shall see later, this sudden change of
behaviour is caused by the supersaturation property of graphs
(cf.~Theorem~\ref{thm:TuranDensity}).
Put differently, there is a qualitative
difference between random restriction sets detecting graphs with enough
edges to force a single $K_r$ to exist and restriction sets detecting graphs
with enough edges to force a positive $K_r$-density, but the value of this density is not of
big influence.

\smallskip

\noindent {\bf Organisation.} The remainder of this paper is organised as
follows.  In Section~\ref{sec:results} we state our results.  In
Section~\ref{sec:det} we then prove Theorem~\ref{thm:TuranExt} and some
general deterministic lower bounds on the number of hyperedges in
Tur\'annical and approximately Tur\'annical hypergraphs. The proofs for our
results concerning random restrictions for general graphs are contained in
Sections~\ref{sec:random:approx} and~\ref{sec:random:exact} and those
concerning random restrictions for random graphs in
Section~\ref{sec:random-random}. In Section~\ref{sec:SharpThresholds} we
argue that the hypergraph property of being Tur\'annical has a sharp
threshold; that is, the threshold determined in one of our main theorems,
Theorem~\ref{thm:random:approx}, is sharp. In
Section~\ref{sec:conc}, finally, we explain how the concept of random
restrictions generalises to other problems besides Tur\'an's theorem. We
provide an outlook on which phenomena may be observed with regard to
questions of this type and the corresponding evolution of random
restrictions, and how they may differ from the Tur\'an case treated in this
paper.

\section{Results}
\label{sec:results}

In this section we give our results. We start with non-global but dense
restrictions and then turn to sparse restrictions. Finally we consider
sparse restrictions for sparse random graphs.

\subsection{Restrictions that are not global}

For completeness, let us start with a formulation of the problem on
non-global restrictions addressed in Theorem~\ref{thm:TuranExt} 
in the hypergraph terms introduced in Definition~\ref{def:Turannical}.
We define $\cI^{(r)}(n,m)=([n],\cE)$ as the $r$-uniform hypergraph with
hyperedges $\cE:=\big\{K\in\tbinom{n}{r} \colon K\cap[m]\neq\emptyset
\big\}$.

\begin{theorem}
  Let $r\ge 3$ and $n$ and $m\le n$ be positive integers.
  \begin{enumerate}[label=\abc]
    \item\label{thm:Intersection:exact}
      The hypergraph $\cI^{(r)}(n,m)$ is Tur\'annical if and only if $n\leq(r-1)m$. 
    \item\label{thm:Intersection:approx}
      For every $\delta>0$ there exists $\eps>0$ such that if
      $n\ge(1+\delta)(r-1)m$, then $\cI^{(r)}(n,m)$ is not $\eps$-Tur\'annical.
  \end{enumerate}
  \label{thm:Intersection}
\end{theorem}
It is easy to deduce  Theorem~\ref{thm:Intersection} from Theorem~\ref{thm:TuranExt}, which
determines the
maximum number of edges of a graph $G$ which is not detected by
$\cI^{(r)}(n,m)$ exactly, also for the case $n>(r-1)m$.
We prove Theorem~\ref{thm:TuranExt} in Section~\ref{sec:det}.

\subsection{Sparse restrictions}

Next we consider sparser hypergraphs. An easy counting argument (which
we defer to Section~\ref{sec:det}) gives the following lower bounds for the
density of Tur\'annical and approximately Tur\'annical hypergraphs.

\begin{proposition}\label{prop:density}
  Let $r\ge 3$ and $n\ge 5$ be integers, let $\eps$ be a real with
  $0<\eps\le1/(2r)$, and let $\cF=([n],\cE)$ be an $r$-uniform hypergraph.
  \begin{enumerate}[label=\abc]
    \item\label{prop:density:exact}
      If $|\cE|<\frac{n(n-1)(n-2)}{r(r-1)^2(r-2)}$ then~$\cF$ is not Tur\'annical.
    \item\label{prop:density:approx}
      If $|\cE|\le(1-r\eps)\frac1{4r}n^2$, then~$\cF$ is not
      $\eps$-Tur\'annical. 
  \end{enumerate}
\end{proposition}

These density bounds are sharp up to constant factors. In fact,
in random $r$-uniform hypergraphs their magnitudes provide thresholds for being Tur\'annical
and approximately Tur\'annical, respectively, as the following two results show. 
We first state the result concerning the threshold for being approximately Tur\'annical.

\begin{theorem}
\label{thm:random:approx}
  For every integer $r\ge 3$ and every $0<\eps\le 1/(2r)$ there
  are $c=c(r,\eps)>0$ and $C=C(r,\eps)>0$ such that for any sequence $p=p_n$ of probabilities
  \begin{equation*}
    \lim_{n\to\infty}
    \Prob\big( \text{$\cR^{(r)}(n,p)$ is $\eps$-Tur\'annical}\,\big) =
    \begin{cases}
      0, & \text{if $p_n\le c n^{2-r}$ for all $n\in\NATS$}, \\ 
      1, & \text{if $p_n\ge C n^{2-r}$ for all $n\in\NATS$}.
   \end{cases}
  \end{equation*}
\end{theorem}
Clearly, a random $r$-uniform hypergraph with hyperedge probability $p=cn^{2-r}$
asymptotically almost surely (a.a.s.) has less than $\frac{3c}{r!}\binom{n}{2}$
hyperedges. Thus part~\ref{prop:density:approx} of Proposition~\ref{prop:density}
does indeed imply the $0$-statement in Theorem~\ref{thm:random:approx}. A proof
of the $1$-statement is provided in Section~\ref{sec:random:approx}.

Using part~\ref{prop:density:exact} of Proposition~\ref{prop:density}, a
similar calculation shows that a random $r$-uniform hypergraph with hyperedge
probability $p=c n^{3-r}$ with $c>0$ sufficiently small is asymptotically almost surely not
Tur\'annical. The corresponding $1$-statement is given in the
following theorem. For the case $r=3$ the threshold probability is a constant,
which we determine precisely.

\begin{theorem}
\label{thm:random:exact}
  For $r=3$ and $p$ constant we have
  \begin{equation*}
    \lim_{n\to\infty}
    \Prob\big( \text{$\cR^{(3)}(n,p)$ is Tur\'annical}\,\big) =
    \begin{cases}
      0, & \text{if $p\le 1/2$}, \\ 
      1, & \text{if $p>1/2$}.
   \end{cases}
  \end{equation*}
  For every integer $r>3$ there are $c=c(r)>0$ and $C=C(r)>0$ such that
  for any sequence $p=p_n$ of probabilities
  \begin{equation*}
    \lim_{n\to\infty}
    \Prob\big( \text{$\cR^{(r)}(n,p)$ is Tur\'annical} \big) =
    \begin{cases}
      0, & \text{if $p_n\le c n^{3-r}$ for all $n\in\NATS$}, \\ 
      1, & \text{if $p_n\ge C n^{3-r}$ for all $n\in\NATS$}.
   \end{cases}
  \end{equation*}
\end{theorem}
This theorem is proven in Section~\ref{sec:random:exact}. As a side remark
we mention that, for its proof we
shall need a structural lemma (Lemma~\ref{lem:structure}) which classifies graphs
with at least $\tur{r}{n}$ edges and has the following direct consequence
which might be of independent interest.
\begin{lemma}
\label{lem:books}
  For every integer $r\ge 3$ and real $\tilde{\eps}>0$ there exists
  $\delta>0$
  such that for all $n$-vertex graphs~$G$ with
  $e(G)>\tur{r}{n}$ one of the the following is true.
  \begin{enumerate}[label=\rom]
    \item Some vertex in~$G$ is contained in at least $\delta n^{r-1}$ copies
      of~$K_r$.
    \item Some edge in~$G$ is contained in at least
    $(1-\tilde{\eps})(n/(r-1))^{r-2}$ copies of~$K_r$.
  \end{enumerate}  
\end{lemma}
An edge contained in $b$ triangles is sometimes called a \emph{book of size
$b$}.
Lemma~\ref{lem:books} in the case $r=3$ thus states that if $e(G)>\tur{3}{n}$
and no vertex of $G$ is contained in many $K_3$-copies, then $G$ contains a book
of size almost $\tfrac{n}{2}$. We remark that Mubayi~\cite{Mub10} recently
showed that for every $\alpha\in(\tfrac{1}{2},1)$, if $G$ has $e(G)>\tur{3}{n}$
and less than $\alpha(1-\alpha)n^2/4-o(n^2)$ triangles, then $G$ contains a book
of size at least $\alpha n/2$. This result is harder, but does not imply
Lemma~\ref{lem:books}.

\smallskip

Finally, it follows from Friedgut's celebrated
result~\cite{Friedgut99:Sharp} that the property of being Tur\'annical
considered in Theorem~\ref{thm:random:exact} has a sharp threshold. This is
detailed in Section~\ref{sec:SharpThresholds}.

\subsection{Sparse restrictions for sparse random graphs}

In the previous subsection we examined the effect of random restrictions on
Tur\'an's theorem. A version of Tur\'an's theorem for the Erd\H{o}s-R\'enyi random graph
$\Gnq$ was recently proved by Schacht~\cite{Schacht:KLR}, and independently by Conlon
and Gowers~\cite{ConGow:KLR}. To understand this theorem, one should view
Tur\'an's theorem as the statement that the fraction of the edges one must delete
from the complete graph $K_n$ to remove all copies of $K_r$ is
approximately $\frac{1}{r-1}$. One can replace $K_n$ with any graph $G$, and
ask which graphs $G$ have the property that deletion of a fraction of
approximately $\frac{1}{r-1}$ of the edges is necessary to remove all copies of $K_r$.

\begin{theorem}[Schacht~\cite{Schacht:KLR}, Conlon \& Gowers~\cite{ConGow:KLR}]
\label{thm:rr:mathias1} Given
  $\eps>0$ and $r$ there exists a constant $C$ such that the following is true.
  For $q\geq Cn^{-2/(r+1)}$, a.a.s.\ $G=\Gnq$ has the property that every
  subgraph of $G$ with at least $(1+\eps)\frac{r-2}{r-1}e(G)$ edges contains a
  copy of $K_r$.
\end{theorem}
Prior to the recent breakthroughs~\cite{Schacht:KLR} and~\cite{ConGow:KLR},
Theorem~\ref{thm:rr:mathias1} was known for $r=3,4,5$
(see~\cite{FrRo:K3K486,KLR,GeSchSte:K5Free}, respectively). We remark that
this is also closely related to the more general line of research
concerning the local and global resilience of graphs, which recently
received increased attention, after the work of Sudakov and
Vu~\cite{SudVuResil}.

Theorem~\ref{thm:rr:mathias1} is best possible in the sense that it ceases
to be true for values of~$q$ growing more slowly than
$n^{-2/(r+1)}$. Moreover, $\eps$ cannot be replaced by~$0$.

Again, the restriction set in Theorem~\ref{thm:rr:mathias1} is the complete
$r$-partite hypergraph (sequence).  So, extending
Theorem~\ref{thm:random:approx}, we would like to analyse what happens when this
is replaced by a sparser set of random restrictions and investigate the influence
of the two independent probability parameters (coming from the random
restrictions and the random graph) on each other. Thus, we will be dealing with
two random objects: namely a random $r$-uniform hypergraph $\cR^{(r)}(n,p)$ and a
random graph $\Gnq$, picked at the same time. Furthermore, since we wish to prove
asymptotically almost sure results, we need to refer not to single $n$-vertex
hypergraphs but to sequences of hypergraphs and graphs.

Before we can formulate our result, we first need to generalise the
concept of being Tur\'annical or approximately Tur\'annical from (copies
of~$K_r$ in) the complete graph~$K_n$ to arbitrary graphs~$G$.  Observe
that, in Theorem~\ref{thm:rr:mathias1} we are interested in graphs~$G$ for
which any subgraph with at least $(1+\eps)\frac{r-2}{r-1}\cdot e(G)$ edges
contains a copy of $K_r$. Hence it is natural to say that the $r$-uniform
hypergraph $\cF$ is $\eps$-Tur\'annical for $G$ when $\cF$ detects every
such subgraph.

For finding a similarly suitable definition of Tur\'annical hypergraphs
for~$G$ we need some additional observations. Recall that $\eps$ cannot
be~$0$ in Theorem~\ref{thm:rr:mathias1}. In other words an \emph{exact}
version of Tur\'an's theorem for random graphs cannot be expressed in terms
of the number of its edges. Instead it has to utilise the structure
provided by Tur\'an's theorem: the maximal $K_r$-free subgraph of $G=\Gnq$
should have exactly as many edges as the biggest $(r-1)$-partite subgraph
of~$G$.  Accordingly, we will call a hypergraph Tur\'annical for~$G$ if it
detects all subgraphs with more edges.
The following definition summarises this.

\begin{definition}[Tur\'annical for $G$]\label{def:Tur:G}
Let $r\geq 3$ be an integer, $G$ an $n$-vertex graph, and $\cF$ an $r$-uniform
hypergraph on the same vertex set.
Then we call $\cF$ \emph{exactly
Tur\'annical for $G$} when the following holds. Every subgraph of $G$ with more
edges than are contained in a maximum $(r-1)$-partition of $G$ has a copy of
$K_r$ induced by an edge of $\cF$.
We say that $\cF$ is \emph{$\eps$-approximately Tur\'annical for $G$}, or simply
\emph{$\eps$-Tur\'annical for $G$}, if every subgraph of $G$ with more than
$(1+\eps)\frac{r-2}{r-1}e(G)$ edges has a copy of $K_r$ induced by an edge of
$\cF$.
\end{definition}

 In this language, Theorem~\ref{thm:rr:mathias1} becomes the statement that,
 given $r$ and $\eps>0$, there exists $C$ such that the complete $r$-uniform
 hypergraph is a.a.s.\ $\eps$-Tur\'annical for $\Gnq$, whenever $q\geq
 Cn^{-2/(r+1)}$. Moreover, according to a result of Brightwell, Panagiotou
 and Steger~\cite{BriPanSte}, for every $r$ there exists $\mu>0$ such that the complete
 $r$-uniform hypergraph is a.a.s.\ exactly Tur\'annical for $\Gnq$ whenever
 $q>n^{-\mu}$.\footnote{However,  Brightwell, Panagiotou
 and Steger do not believe that their result is best possible: for example, 
 for $r=3$ their proof works for $\mu=1/250$, but they suggest the result might
 hold for any $\mu<1/2$.}

In our last theorem we determine the relationship between $r$,
$\eps>0$, $p$ and $q$ such that the random $r$-uniform hypergraph
$\cR^{(r)}(n,p)$ is a.a.s.\ $\eps$-Tur\'annical for $\Gnq$. Not surprisingly,
a suitable combination of the two threshold probabilities from
Theorem~\ref{thm:random:approx} and Theorem~\ref{thm:rr:mathias1}
determines the threshold in this case.

\begin{theorem}\label{thm:TurForGnp} Given $r\in\mathbb N$, $r\ge 3$ and
$\eps\in(0,1/(r-2))$, there exist $c=c(r,\eps)>0$ and
  $C=C(r,\eps)>0$ such that for any pair of sequences $p=p_n$ 
  and $q=q_n$ of probabilities and for
  $\vartheta_q(n):=(nq^{(r+1)/2})^{2-r}$ we have
  \begin{multline*}
    \lim_{n\to\infty}
    \Prob\big( \text{$\cR^{(r)}(n,p)$ is $\eps$-Tur\'annical for
    $\Gnq$}\,\big) \\
    =
    \begin{cases}
      0, & \text{if $p_n\le c \vartheta_q(n)$ for all $n\in\NATS$}, \\ 
      1, & \text{if $p_n\ge C \vartheta_q(n)$ for all $n\in\NATS$}.
   \end{cases}
  \end{multline*}
\end{theorem}

This theorem states that for a fixed~$q_n$ the threshold
probability for $\cR^{(r)}(n,p)$ to be $\eps$-Tur\'annical for $\Gnq$
is~$\vartheta_q(n)$.  Equivalently, if instead we fix the hyperedge
probability~$p_n$ then $\vartheta_p(n):=(np^{1/(r-2)})^{-2/(r+1)}$ is the
threshold probability for $\Gnq$ such that $\cR^{(r)}(n,p)$ is
$\eps$-Tur\'annical for $\Gnq$. In particular, $\vartheta_q(n)$ is constant
when~$q_n$ is the threshold probability from
Theorem~\ref{thm:random:approx} and $\vartheta_p(n)$ is constant when $p_n$ is the
threshold probability from Theorem~\ref{thm:rr:mathias1}.

We note that the requirement $\eps<1/(r-2)$ in Theorem~\ref{thm:TurForGnp} is
necessary for the 0-statement. Indeed, if $\eps>1/(r-2)$ then
$(1+\eps)\frac{r-2}{r-1}e(G)>e(G)$. Therefore the premise in
Definition~\ref{def:Tur:G} is never met, and consequently every hypergraph is
$\eps$-Tur\'annical.

In order to establish Theorem~\ref{thm:TurForGnp} we employ in
Section~\ref{sec:random-random} Schacht's machinery from~\cite{Schacht:KLR}.
However we need to modify this machinery to allow working with two sources of
randomness: graphs $\Gnq$ and hypergraphs $\cR^{(r)}(n,p)$. We believe that this
might prove useful in the future. 


We believe that a similar result as Theorem~\ref{thm:TurForGnp} should be true if $\eps$-Tur\'annical is
replaced by exactly Tur\'annical in this theorem. More precisely, we think that
for $r\geq 3$ the hypergraph $\cR^{(r)}(n,p)$ is a.a.s.\ exactly Tur\'annical for
$\Gnq$, if $p$ and $q$ are both sufficiently large. For obtaining a
result of this type, possibly a modification of the methods used
in~\cite{BriPanSte} may be of assistance.

\section{Deterministic constructions}
\label{sec:det}

In this section we provide the proofs for Theorem~\ref{thm:TuranExt} and
Proposition~\ref{prop:density}. 
We start with the latter.

Let $\cF=(V,\cE)$ be an $r$-uniform hypergraph and~$X$ be a subset of its
vertices of size $|X|=s<r$. The \emph{link hypergraph}
$\link_\cF(X)=(V,\cE')$ of $X$ is the $(r-s)$-uniform hypergraph
with hyperedges $\cE'=\{Y\in\binom V{r-s}\colon Y\cup X\in\cE\}$. If
$X=\{x_1,\dots,x_s\}$ we also write $\link_\cF(x_1,\dots,x_s)$
for $\link_\cF(X)$. When the underlying hypergraph $\cF$ is
clear from the context we write $\link(X)$ instead of
$\link_\cF(X)$.

\begin{proof}[Proof of Proposition~\ref{prop:density}]
  Let the $r$-uniform hypergraph $\cF=([n],\cE)$ be given. 
  We start with the proof of~\ref{prop:density:exact} and first consider the case $r>3$.
  We have
  \begin{equation*}
    \sum_{\{u,v\}\in\binom{[n]}{2}} e\big(\link(u,v)\big)
    =\binom{r}{2}|\cE|
    <\binom{r}{2}\frac{n(n-1)(n-2)}{r(r-1)^2(r-2)}
    \le\frac{\binom{n}{2}n}{(r-2)(r-1)}\,,
  \end{equation*}
  Accordingly there are two vertices $u,v\in[n]$ such that
  $(r-2)e\big(\link(u,v)\big)\le n/(r-1)$.  Let
  \begin{equation*}
    L:=\Big\{ w\in[n]\colon w\in Y 
      \text{ for some } Y\in E\big(\link(u,v)\big) \Big\}
  \end{equation*}
  be the set of vertices covered by the hyperedges of $\link(u,v)$. Because
  $\link(u,v)$ is an $(r-2)$-uniform hypergraph, it follows from the
  choice of~$u$ and~$v$ that $|L|\le n/(r-1)$. 
  Now suppose the graph $G=([n],E)$ is a copy of  the $(r-1)$-partite Tur\'an
  graph $\Tur{r}{n}$ such that ~$u$ and~$v$ are in the same partition
  class of $\Tur{r}{n}$ and $L$ is entirely contained in another 
  partition class. The graph $G$ exists because some partition
  class of $\Tur{r}{n}$ has at least $n/(r-1)$ vertices, and at least two partition classes of
  $\Tur{n}{r}$ have at least two vertices (unless $n\le r$, in which
  case~$L=\emptyset$).   
  As $r>3$, we can add the edge $uv$ to~$G$ without
  creating a copy of~$K_r$ on any hyperedge of~$\cF$.
  Therefore~$G+uv$
  witnesses that~$\cF$ is not Tur\'annical.
  
  For the case $r=3$ of~\ref{prop:density:exact} we proceed similarly and infer
  from $|\cE|<\frac12\binom{n}{3}$ that there are distinct vertices
  $u,v\in[n]$ with $e\big(\link(u,v)\big)<\frac{n}{2}-1$ (observe that the hyperedges
  in $\link(u,v)$ are singletons). Accordingly we can place the vertices $u,v$
  together with $E\big(\link(u,v)\big)$ into one partition class of
  the bipartite graph $\Tur{3}{n}$ and subsequently add the edge
  $uv$. $\cF$ does not detect~$G$, even thought
  $e(G)=\tur{3}{n}+1$.

  For~\ref{prop:density:approx} an even simpler construction for
  $G=([n],E)$ suffices. We start with the complete graph $K_n=:G$. Then,
  for each hyperedge $Y$ of $\cF$ we pick two arbitrary vertices $u,v\in Y$
  and delete the edge $uv$ from $G$ (if it is still present).
  Using $|\cE|\le(1-r\eps)\frac1{4r}n^2$ and $r\ge3$, $n\ge5$, it
  is easy to check that the resulting graph~$G$ has more than
  $(1+\eps)\tur{r}{n}$ edges, and by construction~$G$ contains no
  copies of $K_r$ on hyperedges of~$\cF$. Hence~$\cF$ is not
  $\eps$-Tur\'annical.
\end{proof}

Now we turn to the proof of Theorem~\ref{thm:TuranExt}, which 
provides an upper bound on the number of edges
in a graph on $n$ vertices with the property that no $r$-clique intersects
a fixed set~$M$ of~$m$ vertices.  Theorem~\ref{thm:TuranExt} states that the following
graphs~$\Turm{r}{n}{m}$ are extremal for this problem.  For $n\le(r-1)m$
let $\Turm{r}{n}{m}=\Tur{r}{n}$ be a Tur\'an graph on~$n$ vertices.  For
$n>(r-1)m$ we construct $T=\Turm{r}{n}{m}$ as follows. Initially, we take
$T=\Tur{r}{(r-1)m}$.  We then fix an arbitrary set $M\subset V(T)$ of size
$m$ and add $n-(r-1)m$ new vertices to~$T$.  Finally, for each of the new
vertices we add edges to all other vertices except those in~$M$.  By
construction, it is clear that $\Turm{r}{n}{m}$ has~$n$ vertices and no
copy of $K_r$ intersects $M$. Moreover, observe that the number of edges of
$\Turm{r}{n}{m}$ is given by the function $\turm{r}{n}{m}$ defined
in~\eqref{eq:turm} since
\begin{equation*}
    m^2\tbinom{r-1}{2} + m(r-2)\big(n-(r-1)m\big)
    + \tbinom{n-(r-1)m}{2} =
    \binom{n}{2} - nm + (r-1)\binom{m+1}{2} \,.
\end{equation*}

We shall use the following notation. Let~$G$ be a graph, $X$ and~$Y$ be disjoint
subsets
of its vertices, and~$u$ be a vertex. Then we write $G[X]$ for the subgraph
of~$G$ induced by~$X$ and $G[X,Y]$ for the bipartite subgraph of~$G$ on vertex
set $X\cup Y$ which contains exactly those edges of~$G$ which run between~$X$
and~$Y$. Moreover, we write $\neighbor(u,X)$ for the set of neighbours of~$u$
in~$X$, and set $\deg(u,X):=|\neighbor(u,X)|$.

\begin{proof}[Proof of Theorem~\ref{thm:TuranExt}]
  Let~$r$, $n$, $m$ be fixed and let~$G$ and~$M$ satisfy the conditions of the
  theorem. Assume moreover, that $G$ has a maximum number of edges, subject to
  these conditions. The definition of $\turm{r}{n}{m}$ suggests the following
  case distinction.  We shall first proof the theorem for $n\leq (r-1)m$ and
  then for $n>(r-1)m$. In fact, for the second case we use the correctness of
  the first case.

  \smallskip

  \noindent{\sl{First assume $n\leq (r-1)m$.}} In this case we start
  by iteratively finding vertex disjoint cliques $Q_1,\dots,Q_k$ with
  at least $r$ vertices in~$G$ as follows.  Assume, that
  $Q_1,\dots,Q_{i-1}$ have already been defined for some~$i$. Then let
  $Q_i$ be an arbitrary maximum clique on at least~$r$ vertices in
  $G-\bigcup_{j<i}Q_j$. If no such clique exists, then set $k=i-1$
  and terminate.

  Now, let us establish some simple bounds on the number of edges
  between these cliques and the rest of~$G$. For this purpose, set
  $q_i:=v(Q_i)\ge r$ to be the number of vertices of the clique $Q_i$ for all $i\in[k]$ and $q:=\sum_{i=1}^k q_i$. Clearly,
  the graph $G-\bigcup_{i=1}^k V(Q_i)$ is $K_r$-free, and therefore
  \begin{equation*}
    e\Big(G-\bigcup_{i=1}^k V(Q_i)\Big)\le\tur{r}{n-q}\,.
  \end{equation*}
  Moreover, $M\subset V(G)\setminus \bigcup_{i=1}^k V(Q_i)$ and we
  have $\deg(v,Q_i)\le r-2$ for each $v\in M$, as $v$ is not contained
  in a copy of $K_r$ by assumption. In addition, the maximality of $Q_1,\ldots,Q_k$
  implies that $\deg(v,Q_i)\le q_i-1$ for any
  $v\in V(G)\setminus (M\cup\bigcup_{j=1}^{i}V(Q_i))$. Putting these three 
  estimates together we obtain
  \begin{equation}
  \label{eq:TuranExt:g}
  \begin{split}
    e(G) \le
    \sum_{i=1}^k\binom{q_i}{2} 
    &+
    \sum_{1\le i<j\le k} (q_i-1)q_j +
    \tur{r}{n-q}+mk(r-2) \\
    & +(q-k)(n-m-q)
    =:g(q_1,\ldots,q_k)\,.
  \end{split}
  \end{equation}
  Observe that~\eqref{eq:TuranExt:g} defines a function
  $g(q_1,\dots,q_\ell)$ for each number of arguments~$\ell$.  In
  particular, we also allow $\ell=0$, in which
  case~\eqref{eq:TuranExt:g} asserts that $g()=\tur{r}{n}$.  In the
  remainder of this case of the proof we shall investigate the family
  of functions $g(q_1,\dots,q_\ell)$. We shall show, that for all
  $\ell>0$ we have $g()>g(q_1,\dots,q_\ell)$, which is a consequence
  of the following claim.
  \begin{AuxiliaryCl}\label{cl:PushingVertices}
    Assuming that $q=\sum_{i=1}^k q_i\le n-m$ and $q_i\ge r$ for all
    $i\in[k]$ we have
    \begin{align}
      \label{eq:PushingOne}
      g(q_1,\ldots,q_{k-1},q_k)&< g(q_1,\ldots,q_{k-1},q_k-1) &
      &\text{if $q_k>r$}\,,\quad\text{and} \\
      \label{eq:PushingTwo}
      g(q_1,\ldots,q_{k-1},q_k)&<g(q_1,\ldots,q_{k-1}) &
      &\text{if $q_k=r$}\,.
    \end{align}
  \end{AuxiliaryCl}
  \begin{clproof}[Proof of Claim~\ref{cl:PushingVertices}]
    Adding one or $r$ vertices to a Tur\'an graph $\Tur{r}{n'}$ to
    create a bigger Tur\'an graph and counting the additionally created
    edges gives
    \begin{align}\label{eq:TId1}
      \tur{r}{n'+1}-\tur{r}{n'}
        &=n'-\Big\lfloor\frac{n'}{r-1}\Big\rfloor\;,
      \quad\text{and} \\
      \label{eq:TId2}
      \tur{r}{n'+r}-\tur{r}{n'}&=(r-1)n'+\binom{r}{2} -\Big\lfloor\frac
        {n'+r-1}{r-1}\Big\rfloor\;.
    \end{align}
    Observe that $m>1$, or otherwise $r\le q\le n-1\le (r-1)m-1$ would lead to a
    contradiction. If $q_k>r$ then plugging~\eqref{eq:TId1} (with $n'=n-q$) into
    the definition of~$g$ in~\eqref{eq:TuranExt:g} we obtain
    \begin{equation*}
      g(q_1,\ldots,q_{k-1},q_k-1)-g(q_1,\ldots,q_{k-1},q_k)= m-\Big\lfloor
        \frac{n-q}{r-1}\Big\rfloor-1>0\;,
    \end{equation*}
    proving~\eqref{eq:PushingOne}. Similarly, if $q_k=r$ then~\eqref{eq:TId2}
    implies
    \[g(q_1,\ldots,q_{k-1})-g(q_1,\ldots,q_{k-1},q_k)=m-\Big\lfloor
      \frac{n-q}{r-1}\Big\rfloor-1>0\;,\]
    proving~\eqref{eq:PushingTwo}.
  \end{clproof}
  Clearly, applying Claim~\ref{cl:PushingVertices} for
  sequentially decreasing or discarding the last argument of
  $g(q_1,\ldots,q_\ell)$ gives that
  \begin{equation*}
    g\big(v(Q_1),v(Q_2),\ldots,v(Q_k)\big)=g(q_1,\ldots,q_k)\le g()=\tur{r}{n}\,.
  \end{equation*}
  Moreover, equality holds only when $k=0$, that is, when $G$ does not
  contain any $K_r$.  This proves the theorem in the case $n\le
  (r-1)m$.

  \medskip

  \noindent{\sl Now assume $n>(r-1)m$.}
  Let $X\subset V(G)-M$ be the vertices of $V(G)-M$ which possess at least one
  neighbour in $M$. Let $Y:=V(G)-M-X$.  We start by transforming~$G$ into a
  graph with the same number of edges, which satisfies the assumptions of the
  theorem, and which has the clear structure described in the following claim.

  \begin{AuxiliaryCl}
  \label{cl:PartiteStructure}
    We may assume without loss of generality that
    \begin{enumerate}[label=\abc]
     \item\label{cl:PartiteStructure:b}
        For each $x\in X$ we have $\deg(x)\ge n-m$.
     \item\label{cl:PartiteStructure:c}
        $G[M]$ is a complete $s$-partite graph with parts $M_1,\ldots,M_s$,
        for some $s\leq r-1$. Moreover, $\neighbor(u,X)=\neighbor(u',X)$ for
        all $u,u'\in M_i$ and $1\leq i\leq s$.
      \item\label{cl:PartiteStructure:d}
        $G[X]$ is a complete $t$-partite graph with parts $X_1,\ldots,X_t$,
        for some $t$.
      \item\label{cl:PartiteStructure:e} 
        For each~$M_i$ and~$X_j$ with $i\in[s]$
        and $j\in[t]$, either $G[M_i,X_j]$ is complete or empty, which we denote
        by $M_i\sim X_j$ and $M_i\nsim X_j$, respectively.  For each $i\in
        [s]$ we have $M_i\sim X_j$ for at most $r-2$ values of~$j$.
   \end{enumerate}
  \end{AuxiliaryCl}
  \begin{clproof}[Proof of Claim~\ref{cl:PartiteStructure}]
    To see~\ref{cl:PartiteStructure:b}, observe that, if some $x\in X$ were
    adjacent to fewer than $n-m$ vertices of $G$, then deleting all edges
    adjacent to $x$ and inserting edges from $x$ to all vertices in $X\cup Y$
    (except $x$) would yield a modified graph with no $K_r$ intersecting $M$, and
    with at least as many edges as $G$. Note that $x$ gets removed from the set $X$ of neighbours of $M$ to $Y$  during this modification.

    Now we turn to~\ref{cl:PartiteStructure:c}. Suppose that $u$ and
    $v$ are two non-adjacent vertices of $M$. If $\deg(u)\geq\deg(v)$,
    then the graph $G'$ obtained from $G$ by deleting all edges
    emanating from $v$ and inserting all edges from $v$ to
    $\neighbor(u)$ certainly does not have fewer edges than~$G$, and
    further $G'$ does not have any copy of~$K_r$ intersecting~$M$. Clearly,
    repeating this process for every pair of non-adjacent vertices
    of~$M$ gives a graph with the desired property.

    Applying an analogous process to non-adjacent vertices in~$X$ we
    infer~\ref{cl:PartiteStructure:d}. Note that these deletion and insertion
    processes in~$M$ and~$X$ moreover guarantee the first part
    of~\ref{cl:PartiteStructure:e}. The second part follows since otherwise we
    would obtain a $K_r$ intersecting~$M$.
  \end{clproof}
  
  In the following we assume that~$G$ has the partite structure described in
  Claim~\ref{cl:PartiteStructure} and use it to infer some further properties
  of~$G$ which in turn will allow us to obtain the desired bound on the edges
  in~$G$.
  By \ref{cl:PartiteStructure:b}  of Claim~\ref{cl:PartiteStructure} we
  have $|X_j|+\sum_{i:M_i\nsim X_j}|M_i|\leq m$, and hence
  \begin{equation}\label{eq:previous}
    |X|=\sum_j|X_j|\leq \sum_j\Big(m-\sum_{i:M_i\nsim
        X_j}|M_i|\Big)=\sum_j\sum_{i:M_i\sim X_j}|M_i|\leq(r-2)m\,,
  \end{equation}
  where the last inequality follows from \ref{cl:PartiteStructure:e} of
  Claim~\ref{cl:PartiteStructure}.

  Clearly, this implies $|Y|=n-|X|-|M|\ge n-(r-1)m>0$ which allows us to
  conclude that the inequality in
  Claim~\ref{cl:PartiteStructure}\ref{cl:PartiteStructure:b} is in fact an
  equality: Suppose for contradiction that $\deg(x)\geq n-m+1$ for some $x\in
  X$.  Then we may select any $y\in Y$ and obtain a graph $G'$ by deleting all
  edges incident to $y$ and inserting all edges from $y$ to the neighbours of
  $x$. This graph continues to satisfy the conditions of the theorem and has at
  least one more edge. It follows that for each $x\in X$ we have
  $\deg(x)=n-|M|$.

  For each $i\in [s]$ we also have that 
  $M_i\sim X_j$ for exactly $r-2$ values of $j$ (otherwise we could set all
  vertices of $M_i$ adjacent to $y$ for some $y\in Y$ and gain edges,
  since $|Y|>0$). It follows that in fact equality must hold
  in~\eqref{eq:previous} and hence $|X|=(r-2)m$. 
  This implies that $|X\cup M|=(r-1)m$. Hence we may apply the first case of the
  proof on the graph $G[X\cup M]$ and conclude that $e(G[X\cup M])\le
  \tur{r}{(r-1)m}=m^2\binom{r-1}{2}$. Therefore,
  \begin{align*}
    e(G)&=e(G[X\cup M])+|X||Y|+\tbinom{|Y|}{2}\\
    &\le m^2\tbinom{r-1}{2}+m(r-2)(n-(r-1)m)+\tbinom{n-(r-1)m}{2}
     =\turm{r}{n}{m}\;,
  \end{align*}
  as desired.
\end{proof}

\section{Approximately Tur\'annical random hypergraphs}
\label{sec:random:approx}

In this section we prove Theorem~\ref{thm:random:approx}. As noted in
Section~\ref{sec:intro}, the simple deterministic part~\ref{prop:density:approx}
of Proposition~\ref{prop:density}, that no too sparse hypergraph $\cF$ can be
$\eps$-approximately Tur\'annical, gives the $0$-statement. We therefore
focus on the proof of the $1$-statement. To this end we use the following
theorem of Erd\H{o}s and Simonovits~\cite{ErdSim:Supersaturated}.

\begin{theorem}[Erd\H{o}s \& Simonovits~\cite{ErdSim:Supersaturated}]
  \label{thm:TuranDensity} 
  Given any $r\in\mathbb N$ and $\eps>0$, there exists $\delta>0$
  such that the following is true. If $G$ is any $n$-vertex graph with $e(G)\geq (1+\eps)\tur{r}{n}$, then
  there are at least~$\delta n^r$ copies of $K_r$ in $G$.
\end{theorem}


\begin{proof}[Proof of Theorem~\ref{thm:random:approx}] Given $\eps>0$, by
Theorem~\ref{thm:TuranDensity}, there exists $\delta>0$ such that if $G$ is any
graph with $e(G)\geq (1+\eps)\tur{r}{n}$, then $G$ contains at least~$\delta
n^r$ copies of $K_r$.

Let $p\geq\binom{n}{2}n^{-r}/\delta$.
Given one graph $G$ with at least $\delta n^r$ copies of $K_r$, the probability
that $G$ is not detected by $\cR^{(r)}(n,p)$ is at most
\[(1-p)^{\delta n^r}~.\]
Summing over the at most $2^{\binom{n}{2}}$ such graphs $G$, we see that the
probability that there exists an $n$-vertex graph $G$, with at least $\delta
n^r$ copies of $K_r$, which is undetected by $\cR^{(r)}(n,p)$, is at most
\[2^{\binom{n}{2}}(1-p)^{\delta n^r}<2^{\binom{n}{2}}e^{-p\delta n^r}\leq
2^{\binom{n}{2}}e^{-\binom{n}{2}}~,\] which tends to zero as $n$ tends to
infinity. In particular, with probability tending to $1$, any graph $G$ with
$e(G)\geq (1+\eps)\tur{r}{n}$ is detected by $\cR^{(r)}(n,p)$.
\end{proof}

\section{Exactly Tur\'annical random hypergraphs}
\label{sec:random:exact}

In this section we prove Theorem~\ref{thm:random:exact}.  The $0$-statement
of Theorem~\ref{thm:random:exact} follows from
Proposition~\ref{prop:density}~(a) for $r>3$, and from Lemma~\ref{lem:r=3}
below for $r=3$.

\begin{lemma}\label{lem:r=3}
For $p\le \frac 12$,  we have $\Prob  (\cR^{(3)}(n,p)\mbox{ is
Tur\'annical })=o(1)$.
\end{lemma}
\begin{proof}
  By monotonicity, we may assume that $p=\frac 12$. As in the proof of
  Proposition~\ref{prop:density} it suffices to show that there is a.a.s.\
  a pair of vertices $u,v\in V(\cR^{(3)}(n,p))$ with $e(\link(u,v))\le \frac n2
  -2$ (we remark that the hypergraph $\link(u,v)$ is 1-uniform in this case). So
  choose two arbitrary vertices $u$ and~$v$. Observe that from the properties binomial distribution $\Prob\left(e(\link(u,v)\right)> \frac
  n2 -2)\leq 0.6$, for large enough $n$. Let $\{u_1,v_1\}, \dots ,\{u_{\lfloor
    \frac n2\rfloor},v_{\lfloor \frac n2\rfloor}\}$ be disjoint pairs of
  vertices. Using the independence of the variables $e(\link(u_i,v_i))$, we
  obtain that $\Prob\left(\forall i\colon e(\link (u_i,v_i)> \frac
    n2-2\right)\le 0.6^{\lfloor \frac n2\rfloor }=o(1)$.
\end{proof}

For the
$1$-statement of Theorem~\ref{thm:random:exact} we shall, in
Lemma~\ref{lem:structure}, investigate the structural properties of graphs with more edges than a Tur\'an
graph has, and classify them into three possible categories.
We then treat these
three types of graphs separately, and show for each of them that with high
probability a random restriction hypergraph~$\cR^{(r)}(n,p)$
detects \emph{each} of the graphs of this type.
Let us first take a small detour.

\smallskip


The Erd\H{o}s-Simonovits theorem, Theorem~\ref{thm:TuranDensity}, states that
graphs~$G$ with many more edges than a Tur\'an graph~$\Tur{r}{n}$ contain a
positive fraction of the possible $r$-cliques. This is not true anymore
when~$G$ has just one edge more than~$\Tur{r}{n}$. However, as the well-known
stability theorem of Simonovits~\cite{S68} shows, we can still
draw the same conclusion when we know in addition that~$G$ looks very different
from~$\Tur{r}{n}$. To state the result of Simonovits we need the following
definition. Let~$\eps$ be a positive constant and~$G$ and~$H$ be graphs on $n$ vertices.
If~$G$ cannot be obtained from $H$ by adding and deleting together at most
$\eps n^2$ edges, then we say that~$G$ is \emph{$\eps$-far} from~$H$.

\begin{theorem}[Simonovits~\cite{S68}]
\label{thm:stability}
  For every $r\ge 3$ and $\eps>0$ there exists
  $\delta>0$ such that any $n$-vertex graph $G$ with
  $e(G)\ge\tur{r}{n}$ which is $\eps$-far from~$\Tur{r}{n}$ contains at
  least $\delta n^r$ copies of $K_r$.
\end{theorem}

If a graph~$G$ is not far from a Tur\'an graph, on the other hand, we have a
lot of structural information about~$G$: we know that its vertex set can be
partitioned into $r-1$ sets which are almost of the same size and almost
independent, such that most of the edges between these sets are present. If in
addition almost all vertices of~$G$ have many neighbours in all partition
classes other than their own, then we say that~$G$ has an $\eps$-close $(r-1)$-partition.
The following definition makes this precise. 

\begin{definition}[$\eps$-close $(r-1)$-partition]
\label{def:partition}
  Let~$G=(V,E)$ be a graph. An \emph{$\eps$-close $(r-1)$-partition} of~$G$ is a
  partition $V=V_0\dcup V_1\dcup\dots\dcup V_{r-1}$ of its vertex set such that
  \begin{enumerate}[label=\rom]
    \item\label{def:partition:size} 
      $|V_0|\le\eps^2 n$ and $|V_i|\ge(1-\eps)\frac{n}{r-1}$ for all
      $i\in[r-1]$,
    \item\label{def:partition:deg} 
      for all $v\in V_0$ we have $\deg(v)\le(1-\eps^2)\frac{r-2}{r-1}n$, and
      for all $i,j\in[r-1]$ with $i\neq j$ and for all $v\in V_i$ we have
      $\deg(v,V_j)\ge(1-\eps)|V_j|$.
  \end{enumerate}
  The edges (non-edges) in such a partition that run between two different parts $V_i$
  and $V_j$ with $1\leq i,j\leq r-1$, are called \emph{crossing}, and those that lie
  within a partition class $V_i$ with $1\leq i\leq r-1$, are \emph{non-crossing}.
\end{definition}

The following lemma states that a graph which has at least as many edges as
$\Tur{r}{n}$ either contains a vertex whose neighbourhood has a positive
$K_{r-1}$-density, or has an $\eps$-close $(r-1)$-partition.

\begin{lemma}
\label{lem:structure}
  For every integer $r\ge 3$ and real $0<\eps\leq 1/(16r^2)$ there
  exists a positive constant $\delta$ such that for every
  $n$-vertex graph~$G$ with $e(G)\ge\tur{r}{n}$ one of the the following is true.
  \begin{enumerate}[label=\rom]
    \item\label{lem:structure:Kr} 
      Some vertex in~$G$ is contained in at least $\delta n^{r-1}$ copies
      of~$K_r$.
    \item\label{lem:structure:part} 
      $G$ has an $\eps$-close $(r-1)$-partition.
  \end{enumerate}
\end{lemma}
We postpone the proof of Lemma~\ref{lem:structure} and first sketch that it
implies Lemma~\ref{lem:books}.
\begin{proof}[Proof of
Lemma~\ref{lem:books}] 
Suppose we are given $r$ and $\tilde\eps$. By monotonicity we may assume that
$\tilde\eps<1/16$. Let
$\delta$ be given by Lemma~\ref{lem:structure} with input parameters $r$ and
$\eps:=\tilde\eps/r^2$. By Lemma~\ref{lem:structure} it suffices to show that in 
each $n$-vertex graph $G$ with 
\begin{equation}\label{eq:C1}
e(G)>\tur{r}{n}
\end{equation} which possesses an
$\eps$-close $(r-1)$-partition $V(G)=V_0\dcup V_1\dcup\dots\dcup V_{r-1}$  there is
an edge contained in at least $(1-\tilde\eps)(n/(r-1))^{r-2}$ copies of $K_r$.
First observe that by~\eqref{eq:C1} and~\ref{def:partition:deg} of
Definition~\ref{def:partition} we have $e(G-V_0)>\tur{r}{n-|V_0|}$. Thus, by
Tur\'an's Theorem, there is an edge $uv\subset V_i$ for some $i\in[r-1]$. The
edge $uv$ has at least $(1-2\eps)|V_j|$ common neighbours in each $V_j$, $j\neq i$, creating at least
\[\Big(\big(1-(r-1)\eps\big)(1-\eps)\frac{n}{r-1}\Big)^{r-2}\ge
(1-r\eps)^{r-2}\Big(\frac{n}{r-1}\Big)^{r-2}\ge
(1-\tilde\eps)\Big(\frac{n}{r-1}\Big)^{r-2}\]
copies of~$K_r$.
\end{proof}

\medskip
 \begin{proof}[Proof of Lemma~\ref{lem:structure}]
Given $r$ and $\eps$, let $G$ be an $n$-vertex graph with $e(G)\ge\tur{r}{n}$.
By Theorem~\ref{thm:stability}, there exists $\gamma=\gamma(\eps,r)>0$ such that
if $G$ is $\eps^3/(16r^3)$-far from~$\Tur{r}{n}$, then $G$ contains $\gamma n^r$
copies of $K_r$. We set
\[\delta:=\min\Big\{\gamma,\,\frac{1}{r!2^rr^r},\,\frac{\eps}{4^rr^r},\,\Big(\frac{\eps}{2r}\Big)^{r-1}\Big\}\,.\]
Since $e(G)\ge\tur{r}{n}$, either $G=\Tur{r}{n}$, which clearly has an
$\eps$-close $(r-1)$-partition, or $G$ contains a copy of $K_r$. Observe that the
last term in this minimum ensures that if $n<\tfrac{2r}{\eps}$, then $\delta
n^{r-1}<1$, and thus that one copy of $K_r$ in $G$ is enough to satisfy the
Lemma. It follows that we may henceforth assume $n\ge\tfrac{2r}{\eps}$.

As $G$ contains $\gamma
n^r$ copies of $K_r$ then there
 is a vertex lying in $\gamma n^{r-1}\geq\delta n^{r-1}$ copies of
$K_r$. Thus we may assume that $G$ is not $\eps^3/(16r^3)$-far
from~$\Tur{r}{n}$. So there exists a balanced partition
$V(G)=U_1\dcup\ldots\dcup U_{r-1}$ such that the total number of non-edges
between the parts is at most $\eps^3n^2/(16r^3)$.

Now for each $1\leq i\leq r-1$, we define
\begin{equation}
\label{eq:structure:Vi}
  V_i=\left\{v\in V(G)\colon \deg(v,V(G)\setminus U_i)\geq
  \Big(\frac{r-2}{r-1}-\frac{\eps}{4r}\Big)n\right\}\,.
\end{equation}
We let $V_0:=V(G)\setminus(V_1\cup\ldots\cup V_{r-1})$. We aim to show that
either there is some vertex of $G$ which lies in at least $\delta n^{r-1}$
copies of $K_r$, or that $V_0\dcup V_1\dcup\ldots\dcup V_{r-1}$
is an $\eps$-close $(r-1)$-partition.

For each $1\leq i\leq r-1$, every vertex in $U_i\setminus V_i$ lies in at least
$\eps n/(4r)$ non-edges crossing the partition $(U_1,\ldots,U_{r-1})$. It
follows that 
\begin{equation}\label{eq:UiViSame}
|U_i\setminus V_i|\leq\frac{\eps^2 n}{4r^2}\;,
\end{equation} since there
are at most $\eps^3n^2/(16r^3)$ such non-edges. Summing over $i=1,\ldots,r-1$
we get
\begin{equation}
\label{eq:structure:V0}
  |V_0|\leq \frac{(r-1)\eps^2n}{4r^2}<\frac{\eps^2 n}{4r}<\eps^2 n\,.
\end{equation}
Since $n\ge2r/\eps$ we also have, for each $1\leq i,j\leq r-1$ with $i\neq
j$, and each $v\in V_i$, that
\begin{equation}
\label{eq:structure:degVi}
\begin{aligned}
 |V_i|&\geq |U_i|-\frac{\eps^2 n}{4r^2}>(1-\eps)\frac{n}{r-1}\,,
  \qquad\text{and} \\
  \deg(v,V_j)&\geBy{\eqref{eq:structure:Vi},\eqref{eq:UiViSame}}
  |U_j|-1-\frac{\eps n}{4r}-\frac{\eps^2 n}{4r^2} \\ 
  &\geq |V_j|-1-(r-2)\frac{\eps^2 n}{4r^2}-\frac{\eps n}{4r} -\frac{\eps^2 n}{4r^2}
  \geq(1-\eps)|V_j|\,,
\end{aligned}
\end{equation}  
where we use $\eps\le\frac1{10}$ to obtain the last inequality.

We claim that a vertex $u$ lying in more than one of the sets
$V_1,\ldots,V_{r-1}$ must lie in at least $\delta n^{r-1}$ copies of $K_r$. To
see this, observe that $u$ must have at least $(1-\eps)|V_i|$ neighbours in
$V_i$ for each $1\leq i\leq r-1$. Now consider the following method of
constructing a copy of $K_r$ in $G$ using $u$. We choose a neighbour $v_1$ of
$u$ in $V_1$, a common neighbour $v_2$ of $u$ and $v_1$ in $V_2$, and so
on. Since $\eps\le1/(16r)$, the common neighbourhood of $u,v_1,\ldots,v_{i-1}$
in $V_i$ contains at least $(1-i\eps)|V_i|>\frac{n}{2(r-1)}$ vertices for each $i$,
there are at least $\frac{n}{2(r-1)}$ choices at each of the $r-1$ steps (and in
particular this construction is possible).  This procedure may construct the
same copy of $K_r$ more than once (since at this point we do not yet know that
the sets $V_1,\ldots,V_{r-1}$ are disjoint), but not more than $(r-1)!$
times. It follows that $u$ lies in at
least \[\frac{1}{(r-1)!}\left(\frac{n}{2(r-1)}\right)^{r-1}\geq\delta n^{r-1}\]
copies of $K_r$.

Hence, we can assume from now on that the sets $V_1,\ldots,V_{r-1}$ are
disjoint.
Next we claim that a vertex $u$ in $V_0$ whose degree exceeds
$(1-\eps^2)\frac{r-2}{r-1}n$ must lie in at least $\delta n^{r-1}$ copies of
$K_{r}$. Without loss of generality, we may assume that we have
$\deg(u,V_1)\leq\deg(u,V_2)\leq\ldots\leq\deg(u,V_{r-1})$. Since $u\notin V_1$,
we have
\begin{equation}\begin{split}
\label{eq:structure:degV0V1}
  \deg(u,V_1) &=\deg(u)-\deg(u,V(G)\setminus V_1)  \\
  &\geq\deg(u)-\deg(u,U_2\dcup\ldots\dcup U_{r-1})-|U_1\setminus V_1| \\
  &\gBy{\eqref{eq:structure:Vi},\eqref{eq:UiViSame}}
    (1-\eps^2)\frac{r-2}{r-1}n-\Big(\frac{r-2}{r-1}
    -\frac{\eps}{4r}\Big)n-\frac{\eps^2 n}{4r^2} \\
  &\geq -\eps^2n+\frac{\eps n}{4r}-\frac{\eps^2 n}{4r^2}
  \geq\frac{\eps n}{16r}\,,
\end{split}\end{equation}
where the last inequality follows from $\eps\le1/(16r)$.  Since
$\deg(u,V_2)\geq\deg(u,V_1)$ and $u$ has at most $\frac{n}{r-1}+\eps^2n$
non-neighbours by assumption, we infer that 
$\deg(u,V_2)\geq\frac{n}{3(r-1)}$, using again $\eps\le1/(16r)$.
Hence
\begin{equation}
\label{eq:structure:degV0Vi}
 \deg(u,V_i)\geq\frac{n}{3(r-1)} \qquad\text{for each $2\leq i\leq r-1$.} 
\end{equation}
Now consider the same inductive construction of copies of $K_r$
containing $u$ as before. This time we know that there are at least $\frac{\eps
  n}{16r}$ choices for $v_1$, and at least
\[\frac{n}{3(r-1)}-(i-1)\eps|V_i|>\frac{n}{4(r-1)}\]
choices for $v_i$, for each
$2\leq i\leq r-1$. Since the sets $V_1,\ldots,V_{r-1}$ are disjoint, each copy
of $K_r$ can be constructed in only one way. Thus $u$ does indeed lie in at
least
\[\frac{\eps n}{16r}\left(\frac{n}{4(r-1)}\right)^{r-2}\geq\delta n^{r-1}\]
copies of $K_r$.

Accordingly, we can assume that $\deg(u)\le(1-\eps^2)\frac{r-2}{r-1}n$, for all
$u$ in~$V_0$. Together with~\eqref{eq:structure:V0}
and~\eqref{eq:structure:degVi} this implies that the partition
$V_0\dcup\ldots\dcup V_{r-1}$ satisfies~\ref{def:partition:size}
and~\ref{def:partition:deg} of Definition~\ref{def:partition} and hence is an
$\eps$-close $(r-1)$-partition of~$G$.
\end{proof}

We need a more precise structural result to handle the case $r=3$ of
Theorem~\ref{thm:random:exact}. As we shall see, this is a simple consequence of the above
proof.

\begin{corollary}
\label{cor:3structure}
  For every  $0<\eps\leq1/144$ there exists a positive constant
  $\delta$ such that for all $n$-vertex graphs~$G$ with
  $e(G)\ge\tur{3}{n}$ one of the the following is true.
  \begin{enumerate}[label=\rom]
    \item\label{cor:3structure:tri}
      $G$ contains at least $\delta n^3$ triangles.
    \item\label{cor:3structure:bip}
      There is a vertex $u$ of $G$ such that $\neighbor(u)\supset
      X\dcup Y$, where $|X||Y|\geq \eps n^2/288$ and $e(X,Y)\geq (1-4\eps)|X||Y|$.
    \item\label{cor:3structure:part}
      $G$ has an $\eps$-close $2$-partition.
  \end{enumerate}
\end{corollary}
\begin{proof} 
  We follow the previous proof with $r=3$, using the same value for $\delta$.
  If~$G$ contains less than $\delta n^3$ triangles we obtain the three sets
  $V_0,V_1,V_2$ (as defined in~\eqref{eq:structure:Vi}). If these sets do not
  form a partition of $V(G)$, then there is a vertex $v$ in both $V_1$ and
  $V_2$. Then we let $X:=\neighbor(v)\cap V_1$ and $Y:=\neighbor(v)\cap V_2$.
  By~\eqref{eq:structure:degVi} we have $|X||Y|\geq
  (1-\eps)^2|V_1||V_2|\ge(1-\eps)^4n^2/4>\eps n^2/32$ because $\eps\le1/2$.
  Since each vertex of $X$ is adjacent to all but at most $\eps |V_2|$ vertices
  of $Y$ by~\eqref{eq:structure:degVi}, we also have $e(X,Y)\geq(1-4\eps)|X||Y|$
  as required.

  Hence we may assume that $V_0,V_1,V_2$ form a partition of $V(G)$. The only
  remaining barrier to $V_0,V_1,V_2$ being an $\eps$-close $2$-partition of $G$
  is the existence of a vertex $v$ in $V_0$ with degree more than
  $(1-\eps^2)\frac{n}{2}$. As in the previous proof, if this vertex exists we
  may without loss of generality presume by~\eqref{eq:structure:degV0V1} that it
  has at least $\eps n/48$ neighbours in $V_1$, and
  by~\eqref{eq:structure:degV0Vi} that it has at least $n/6$ neighbours in
  $V_2$. Again we let $X:=\neighbor(v)\cap V_1$, and $Y:=\neighbor(v)\cap V_2$,
  and get $|X||Y|\geq \eps n^2/288$ as required. Now since $|Y|>|V_2|/4$, and
  since every vertex in $X$ is adjacent to all but at most $\eps|V_2|$ vertices
  of $Y$, we have $e(X,Y)\geq (1-4\eps)|X||Y|$ as required.
\end{proof}

Our next lemma counts the number of graphs with an $\eps$-close
$(r-1)$-partition and a given number of non-crossing edges. In addition it estimates the number of
$r$-cliques in such a graph.

\begin{lemma}
\label{lem:counting}
  Let $\l\ge 0$ and $r\ge 3$ be integers, $0<\eps<1/(2r)$ be a real
  and $n\ge2r^3/\eps^2$ be an integer. Let~$\cG$
  be the family of all graphs on a fixed vertex set of size $n$ with
  $e(G)>\tur{r}{n}$ which have an $\eps$-close $(r-1)$-partition with exactly $\l$ non-crossing edges. Then
  \begin{enumerate}[label=\abc]
    \item\label{lem:counting:lPos}if $\l=0$ then $|\cG|=0$,
    \item\label{lem:counting:size} $|\cG|\le r^{5 \l n}$, and
    \item\label{lem:counting:Kr} every $G\in\cG$ contains at least
      $\l\big(\frac{n}{2r-2}\big)^{r-2}$ copies of $K_r$.
 \end{enumerate}
\end{lemma}
\begin{proof}
In the following, let $G\in\cG$. We fix an $\eps$-close $(r-1)$-partition
$V_0,\ldots,V_{r-1}$ of $G$ with $\l$ non-crossing edges. Let the number of
crossing non-edges be $k$.

First we show~\ref{lem:counting:Kr}. Let $e$ be a non-crossing edge of
$G$. Without loss of generality, we may presume $e$ lies in $V_1$. We can
construct an $r$-clique using $e$ as follows: we choose any common neighbour
$v_2$ of $e$ in $V_2$, then a common neighbour $v_3$ of $e$ and $v_2$ in $V_3$, and so on. By
definition of an $\eps$-close $(r-1)$-partition, for each $2\leq i\leq r-1$, the
common neighbourhood of $e,v_2,\ldots,v_{i-1}$ in $V_i$ has size at least
$(1-i\eps)|V_i|>\frac12n/(r-1)$ because $\eps<1/(2r)$. It follows that $e$ lies
in at least $(n/(2r-2))^{r-2}$ copies of $K_r$ in $G$. Further, if $e'$ is a second
non-crossing edge of $G$, then no $r$-clique of $G$ using~$e'$ can be one of the
$r$-cliques through $e$ given by the above construction.
It follows that $G$ contains $\l(n/(2r-2))^{r-2}$ copies of~$K_r$.

Now we prove~\ref{lem:counting:lPos} and~\ref{lem:counting:size}. We first show that
\begin{equation}
\label{eq:counting:lk}
  \l\geq |V_0|+k+1\,.
\end{equation}
If $V_0=\emptyset$, then we have $\tur{r}{n}+1\leq e(G)\leq\tur{r}{n}+\l-k$, and
therefore $\l\geq |V_0|+k+1$. If $V_0\neq\emptyset$ on the other hand, then,
since every vertex in $V_0$ has degree at most $(1-\eps^2)\frac{r-2}{r-1}n$,
we have 
\[\tur{r}{n}+1\leq e(G)\leq
(1-\eps^2)\frac{r-2}{r-1}n|V_0|+\left(\frac{n-|V_0|}{r-1}\right)^2\binom{r-1}{2}+\l-k\,.\]
Using the facts $|V_0|\leq\eps^2n$ and
$\big(\frac{n}{r-1}\big)^2\binom{r-1}{2}\leq \tur{r}{n}+r^2$, we infer
\begin{equation*}\begin{split}
  \tur{r}{n} & +1 \\
  &\leq (1-\eps^2)\tfrac{r-2}{r-1}n|V_0|
    +\big(\tfrac{n}{r-1}\big)^2\tbinom{r-1}{2}
    -\tfrac{r-2}{r-1}n|V_0|+\tfrac{(r-2)}{2(r-1)}|V_0|^2+\l-k \\
  &
  \leq\tur{r}{n}+r^2-\eps^2\tfrac{r-2}{r-1}n|V_0|+\eps^2\tfrac{r-2}{2(r-1)}n|V_0|+\l-k
 \\
  &=\tur{r}{n}+r^2-\eps^2\tfrac{r-2}{2(r-1)}n|V_0|+\l-k\,.
\end{split}\end{equation*}
It follows from $n\ge2r^3/\eps^2$ that
$\eps^2\frac{r-2}{2(r-1)}n|V_0|\geq r^2+|V_0|$, and so we again obtain
$\l\geq |V_0|+k+1$.

Now, if $G\in\cG$ exists, then~\eqref{eq:counting:lk} clearly implies $\l>0$,
proving~\ref{lem:counting:lPos}.
It remains to show~\ref{lem:counting:size}. We can construct any graph $G$ in
$\cG$ as follows. 
We choose $k\in\{0,\ldots,\ell-1\}$. We partition  $[n]$ into
$r$ sets $V_0,\ldots,V_{r-1}$ such that $V_0$
satisfies~\eqref{eq:counting:lk}. For each pair of vertices intersecting $V_0$, we
choose whether or not to make it an edge of $G$; there are at
most $2^{|V_0|n}\le 2^{\ell n}$ such choices.
Then we choose  $k$ pairs of vertices crossing the partition to be
non-edges of $G$, and make all other crossing pairs edges of $G$. Finally, we
choose $\l$ pairs of vertices within partition classes to be the $\l$
non-crossing edges of $G$.
The total number of choices in this process is at most
\[
  \sum_{0\le k\le\l-1}
  r^n2^{\ell n}\binom{\binom{n}{2}}{k}\binom{\binom{n}{2}}{\l}
  \leByRef{eq:counting:lk} \l r^n2^{\l n}n^{2\l+2\l}\leq r^{5\l n}\,,
\]
as required.
\end{proof}

With these tools at hand we can proceed to the proof of
Theorem~\ref{thm:random:exact}. For a binomially distributed random variable~$X$
we will use the following Chernoff bound which can be found, e.g., in~\cite
[Theorem~2.1]{JaLuRu:Book}. For each $\gamma\in(0,\frac13)$ we have
\begin{equation}\label{eq:chernoff}
  \Prob\big(X\le(1-\gamma)\Exp X\big)\le\exp(-\gamma^2\Exp X/2)\;.
\end{equation}

\begin{proof}[Proof of the $1$-statements of Theorem~\ref{thm:random:exact}]
  We shall first prove the case $r=3$ and then turn to the case $r>3$.  In both
  cases we will consider the class~$\cG_r$ of all $n$-vertex graphs $G$ with
  $e(G)>\tur{r}{n}$.  In the case $r=3$, $\cG_3$ can be written as the union of
  three sub-classes $\cG_\mathrm{A}$, $\cG_\mathrm{B}$, and $\cG_\mathrm{C}$
  defined by the properties in~\ref{cor:3structure:tri},
  \ref{cor:3structure:bip}, and~\ref{cor:3structure:part} of
  Corollary~\ref{cor:3structure}, respectively. Similarly, for $r>3$
  Lemma~\ref{lem:structure} allows us to write
  $\cG_r=\cG_\mathrm{D}\cup\cG_\mathrm{E}$, where the graphs~$\cG_{\mathrm{D}}$
  and~$\cG_\mathrm{E}$ enjoy properties given by
  Lemma~\ref{lem:structure}\ref{lem:structure:Kr} and
  Lemma~\ref{lem:structure}\ref{lem:structure:part}, respectively.  We will
  prove that for each of these sub-classes a.a.s.\ the random hypergraph
  $\cR^{(r)}(n,p)$ with~$p$ as required detects all graphs in this
  sub-class. The result then follows from the union bound.

  \medskip

  \noindent{\sl \underline{Case $r=3$:}}
  Let $p>1/2$ be fixed and set
  \begin{equation*}
    \eps:=\min\Big\{
    \frac{1}{144},\frac{p}{8},\frac{2p-1}{4p+3}\Big\}\,.
  \end{equation*}
  Let $\delta>0$ be guaranteed by Corollary~\ref{cor:3structure} for this~$\eps$. Observe that this choice
  of~$\eps$ and~$n$ allows the application of
  Corollary~\ref{cor:3structure}. Further, let
  $\cG_3=\cG_\mathrm{A}\cup\cG_\mathrm{B}\cup\cG_\mathrm{C}$ be as defined
  above.  We will now show for each of the graph classes $\cG_\mathrm{A}$,
  $\cG_\mathrm{B}$, and $\cG_\mathrm{C}$ that a.a.s.\ $\cR^{(3)}(n,p)$ detects
  all their members.

  \smallskip

  Suppose a graph $G\in\cG_\mathrm{A}$ is given. Then
  Corollary~\ref{cor:3structure}\ref{cor:3structure:tri} the graph $G$ contains
  at least $\delta n^3$ triangles. The probability that $\cR^{(3)}(n,p)$ does
  not detect $G$ is at most \[(1-p)^{\delta n^3}\leq e^{-p\delta n^3}\leq
  e^{-\delta n^3/2}\;,\] and since $|\cG_\mathrm{A}|<2^{\binom{n}{2}}$, applying
  the union bound, the probability that there is a graph in $\cG_\mathrm{A}$
  which $\cR^{(3)}(n,p)$ does not detect is at most
  \[2^{\binom{n}{2}}e^{-\delta n^3/2}\;,\]
  which tends to zero as $n$ tends to infinity.

  \smallskip

  Recall that $\cG_\mathrm{B}$ is the sub-class of $\cG_3$ with graphs in which
  there is a vertex $u$ and disjoint set $X,Y\subset \neighbor(u)$ with both
  $|X||Y|\geq\eps n^2/288$ and $e(X,Y)\geq (1-4\eps)|X||Y|$.
  Suppose that a $3$-uniform $n$-vertex hypergraph~$\cH$ has the property that
  for every vertex $v$ and disjoint sets $W$ and $Z$ with $|W||Z|\geq \eps
  n^2/288$, there are more than $4\eps|W||Z|$ hyperedges of $\cH$, each consisting of
  $v$, a vertex of $W$, and a vertex of $Z$. Then, clearly for any
  $G\in\cG_\mathrm{B}$ the hypergraph $\cH$ detects $G$. Hence it remains to
  show that a.a.s.\ $\cR^{(3)}(n,p)$ has this property.

  Given one vertex $v$ and pair of disjoint vertex sets $X$ and $Y$ of
  $\cR^{(3)}(n,p)$ with $|X||Y|\ge \eps n^2/288$ the expected size of
  $E\big(\link_{\cR^{(3)}(n,p)}(v)\big)\cap(X\times Y)$ in $\cR^{(3)}(n,p)$ is
  $p|X||Y|$. Using the Chernoff bound~\eqref{eq:chernoff}, the probability that we have
  \[e\big(\link_{\cR^{(3)}(n,p)}(v)\cap(X\times Y)\big)<4\eps|X||Y|\leq
  p|X||Y|/2\] is at most $e^{-p|X||Y|/8}\le e^{-\eps n^2/5000}$. By the union bound, the probability that
  there exists any such vertex and pair of disjoint subsets in $\cR^{(3)}(n,p)$ is
  at most
  \[n2^n2^ne^{-\eps n^2/5000}\] 
  which tends to zero as $n$ tends to infinity.

  \smallskip

  Finally, $\cG_\mathrm{C}$ is the class of $n$-vertex graphs $G\in\cG_3$ which
  possess an $\eps$-close $2$-partition $V_0\dcup V_1\dcup V_2$.  Since
  $e(G)\ge\tur{r}{n}+1$ there is at least one non-crossing edge $e$ in this
  partition by Lemma~\ref{lem:counting}\ref{lem:counting:lPos}. Without loss of
  generality, we may presume $e$ lies in $V_1$. Then the common neighbourhood of
  $e$ contains more than $(1-2\eps)|V_2|\geq (1-3\eps)\frac{n}{2}$ vertices. In
  particular, if $\cR^{(3)}(n,p)$ has the property that every pair of vertices
  is in at least $(1+3\eps)\frac{n}{2}$ hyperedges, then $\cR^{(3)}(n,p)$
  detects every graph in $\cG_\mathrm{C}$. We will show that a.a.s.\
  $\cR^{(3)}(n,p)$ has this property.

  Given one pair of vertices $u,v$, we
  have $$\Exp\big(e(\link_{\cR^{(3)}(n,p)}(u,v))\big)=p(n-2)\;.$$ Using the fact
  that $\eps\leq\frac{2p-1}{4p+3}$ we note that \[ (1+3\eps)\frac{n}{2}
  \leq\Big(1+3\frac{2p-1}{4p+3}\Big)\frac{n}2 =\Big(1-2\frac{2p-1}{4p+3}\Big)pn
  < (1-\eps)p(n-2)\;,
  \] for large enough $n$.
  The Chernoff bound~\eqref{eq:chernoff}  then gives
  \begin{align*}
    \Prob&\left(e(\link_{\cR^{(3)}(n,p)}(u,v))\le (1+3\eps)\frac{n}{2}\right)\le\\
    &\Prob\left(e(\link_{\cR^{(3)}(n,p)}(u,v))\le (1-\eps)p(n-2)\right)\le e^{-\eps^2p(n-2)/2}\;.
  \end{align*}
  By the union bound, the probability that there exists any such pair of
  vertices in $\cR^{(3)}(n,p)$ is at most $\binom{n}{2}e^{-\eps^2p(n-2)/2}$,
  which tends to zero as $n$ tends to infinity.

  \medskip

  \noindent{\sl \underline{Case $r>3$:}}
  Let $\eps:=1/(16r^2)$, and let $\delta>0$ be the positive constant guaranteed
  by Lemma~\ref{lem:structure} for this $\eps$. 
  Let
  $\cG_r=\cG_\mathrm{D}\cup\cG_\mathrm{E}$ be classes of $n$-vertex graphs
  satisfying~\ref{lem:structure:Kr} and~\ref{lem:structure:part} of
  Lemma~\ref{lem:structure}, respectively.  Set
  \begin{equation*} 
    C:=\max\left\{\frac1\delta,6r(2r-2)^{r-2}\right\}\,,
    \qquad\text{and let}\quad p\geq Cn^{3-r}\,.
  \end{equation*}
  Again, we will prove
  that a.a.s.\ $\cR^{(r)}(n,p)$ detects all graphs in
  $\cG_\mathrm{D}$ and $\cG_\mathrm{E}$.

  \smallskip

  The class $\cG_\mathrm{D}$ contains the graphs from $\cG_r$ in which there is
  a vertex contained in at least $\delta n^{r-1}$ copies of $K_r$. Given one
  such graph $G$, the probability that $G$ is not detected by $\cR^{(r)}(n,p)$
  is at most
  \[(1-p)^{\delta n^{r-1}}<e^{-Cn^{3-r}\delta n^{r-1}}=e^{-C\delta n^2}\leq
  e^{-n^2}\,,\] and since there are at most $2^{\binom{n}{2}}$ graphs in
  $\cG_\mathrm{D}$, the probability that there is a graph in $\cG_\mathrm{D}$
  undetected by $\cR^{(r)}(n,p)$ is at most
  \[2^{\binom{n}{2}}e^{-n^2}\;,\] which tends to zero as $n$ tends
  to infinity.

  \smallskip

  It remains to consider the class $\cG_\mathrm{E}$ of graphs $G\in\cG_r$ with
  $\eps$-close $(r-1)$-partition. For~$1\leq\l\leq\binom{n}{2}$ let
  $\cG_\mathrm{E}(\l)\subseteq \cG_\mathrm{E}$ be the class of graphs that have
  an $\eps$-close $(r-1)$-partition with exactly $\l$ non-crossing edges. By Lemma~\ref{lem:counting}\ref{lem:counting:lPos} we have
  \begin{equation}\label{eq:unioniseverything}
  \bigcup_{1\le\l\le\binom{n}{2}}\cG_\mathrm{E}(\l)=\cG_\mathrm{E}\,.
  \end{equation}
  Now fix~$\l\in\{1,\ldots,\binom{n}{2}\}$.
  Lemma~\ref{lem:counting}\ref{lem:counting:size} asserts that
  $|\cG_\mathrm{E}(\l)|\le r^{5\l n}$.  Moreover, each graph in
  $\cG_\mathrm{E}(\l)$ contains at least $\l(n/(2r-2))^{r-2}$ copies of $K_r$ by
  Lemma~\ref{lem:counting}\ref{lem:counting:Kr}. Hence, by the union bound, the
  probability that $\cR^{(r)}(n,p)$ fails to detect at least one graph in
  $\cG_\mathrm{E}(\l)$ is at most
  \begin{equation*}\begin{split}
      r^{5\l n}(1-p)^{\left(\frac{n}{2r-2}\right)^{r-2}\l}
      & <r^{5\l n}\exp\Big({-Cn^{3-r}\l\Big(\frac{n}{2r-2}\Big)^{r-2}}\Big) \\
      & \leq r^{5\l n}e^{-6r\l n}
      <e^{-\l n}\,.
  \end{split}\end{equation*} 
  Finally, applying the union bound in conjunction
  with~\eqref{eq:unioniseverything}, we conclude that $\cR^{(r)}(n,p)$ detects
  all graphs in~$\cG_\mathrm{E}$ with probability at least $1-\binom{n}{2}e^{-n}$, which tends to one as $n$ tends to infinity.
\end{proof}

\section{Tur\'annical hypergraphs for random graphs}
\label{sec:random-random}

In this section we prove Theorem~\ref{thm:TurForGnp}.  For this purpose we
shall use the machinery developed by Schacht~\cite{Schacht:KLR} for proving
Theorem~\ref{thm:rr:mathias1}. Conlon and Gowers~\cite{ConGow:KLR} obtained independently (using different methods) a result very similar to Schacht's. While either result is equally suited for proving~\ref{thm:TurForGnp} we follow notation introduced in~\cite{Schacht:KLR}. Schacht formulates a powerful abstract
result, a so-called \emph{transference theorem} (Theorem~3.3 in~\cite{Schacht:KLR}; see also Theorem~4.5 in~\cite{ConGow:KLR}),
which is phrased in the language of hypergraphs and gives very general conditions under which a
result from extremal combinatorics may be transferred to an analogue for
sparse random structures. Actually, Theorem~\ref{thm:rr:mathias1} mentioned
above is only one of several results where the transference theorem applies.
Schacht, and Conlon and Gowers, give further applications to transfer the
multidimensional Szemer\'edi theorem, a result on Schur's equation, and others. Here we are
interested in a transference of Theorem~\ref{thm:random:approx}.

Below we will state a special version of Schacht's transference theorem,
tailored to our situation. For formulating this theorem we need some
definitions. We remark that in these definitions we slightly deviate from
Schacht's setting. More precisely, the transference theorem uses a certain
sequence of hypergraphs which encode the classical extremal problem under
consideration. In the case of Tur\'an's problem for~$K_r$, the $n$-th
hypergraph in this sequence has vertex set $E(K_n)$ and a hyperedge for every
$\binom{r}{2}$-tuple of elements from $E(K_n)$ which form a copy of~$K_r$
in~$K_n$ in Schacht's setting. Instead, we shall work with $r$-uniform
hypergraphs~$\cH_n$ on vertex set $V(K_n)$, making use of the fact that a
copy of $K_r$ is uniquely identified by its vertices.  The corresponding
modifications of the definitions and of the transference theorem are
straightforward.

The transference theorem requires the sequence of hypergraphs to satisfy
two conditions. The first one is a requirement upon the extremal
problem to be transferred, namely, that it has a certain `super-saturation' property (similar
to the one given in Theorem~\ref{thm:TuranDensity}).
The following definition makes this precise.

\begin{definition}[$(\alpha,\eps,\zeta)$-dense] 
\label{def:dense}
  Let $\mathbf{H}=(\cH_n)_{n\in\mathbb{N}}$ be a sequence of $n$-vertex
  $r$-uniform hypergraphs, $\alpha\geq 0$ and $\eps,\zeta>0$ be constants. We
  say $\mathbf{H}$ is \emph{$(\alpha,\eps,\zeta)$-dense} if the following is
  true. There exists $n_0$ such that for every $n\geq n_0$
  and every graph $G$ on the vertex set $V(\cH_n)$ with at least
  $(\alpha+\eps)\binom{n}{2}$ edges, the number of copies of $K_r$ in $G$ induced by hyperedges of $\cH_n$
  is at least $\zeta e(\cH_n)$.
\end{definition}

The second condition 
determines the sparseness of a random graph to which one may transfer the
extremal result.
Given an $r$-uniform hypergraph $\cH$, a graph $G$ on the same
vertex set, and a pair of distinct vertices $u$ and $v$ of $V(G)$, we let $\deg_i(u,v,G)$
be the number of hyperedges of $\cH$ containing $u$, $v$ and at least $i$ edges of
$G$, not counting the possible edge $uv$. If $u=v$ we let $\deg_i(u,v,G):=0$. The hypergraph $\cH$ itself is suppressed from the notation as it will be clear from the context.
We set 
\[\mu_i(\cH,q):=\mathbb{E}\Big[\sum\nolimits_{u,v}\deg_i^2\big(u,v,\Gnq\big)\Big]\,,\]
where the expectation is taken over the space of random graphs $\Gnq$.

\begin{definition}[$(K,\mathbf{q})$-bounded]\label{def:bounded}
  Let $\mathbf{H}=(\cH_n)_{n\in\mathbb{N}}$ be a sequence of $n$-vertex
  $r$-uniform hypergraphs, $\mathbf{q}=(q_n)_{n\in\mathbb{N}}$ be a
  sequence of probabilities, and $K\geq 1$ be a constant. We say that
  $\mathbf{H}$ is \emph{$(K,\mathbf{q})$-bounded} if the following holds.
  For each $i\in[\binom{r}{2}-1]$ there exists $n_0$ such that for each
  $n\geq n_0$ and $q\geq q_n$ we have
  \[\mu_i(\cH_n,q)\leq Kq^{2i}\cdot\frac{e(\cH_n)^2}{n^2}\,.\]
\end{definition}

We can now state (a special case of) Schacht's transference theorem.

\begin{theorem}[transference theorem, Schacht~\cite{Schacht:KLR}]\label{thm:rr:mathias2}
  For all $r\geq 3$, $K\ge 1$, $\delta>0$, $\zeta>0$ and
  $(\omega_n)_{n\in\NATS}$ with $\omega_n\to\infty$ as $n\to\infty$, there
  exists $C>1$ such that the following holds. Let $\eps:=8^{-r(r-1)/2}\delta$,
  and let $\mathbf{H}=(\cH_n)_{n\in\mathbb{N}}$ be a sequence of $n$-vertex
  $r$-uniform hypergraphs which is $(\frac{r-2}{r-1},\eps,\zeta)$-dense.  Let
  $\mathbf{q}=(q_n)_{n\in\mathbb{N}}$ be a sequence of probabilities with
  $q_n^{r(r-1)/2}\cdot e(\cH_n)\rightarrow\infty$ and $Cq_n<1/\omega_n$ such
  that $\mathbf{H}$ is $(K,\mathbf{q})$-bounded.
  
  Then the following holds a.a.s.\ for $G=\Gnq[Cq_n]$. Every subgraph of~$G$
  with at least $(\frac{r-2}{r-1}+\delta)\cdot e(G)$ edges contains an
  $r$-clique induced by a hyperedge of $\cH_n$.
\end{theorem}

We remark that the quantification in this theorem and the
$(\alpha,\eps,\zeta)$-dense\-ness condition given here is not the same as
in~\cite{Schacht:KLR} (in fact, in~\cite{Schacht:KLR} the two parameters $\eps$
and $\zeta$ are not made explicit in the concept of $\alpha$-dense\-ness used
in~\cite{Schacht:KLR}). The statement in~\cite{Schacht:KLR} is certainly cleaner, but for
our purposes it is necessary that we check the denseness condition only for a
special $\eps$ (as opposed to all $\eps>0$, which is necessary for
the original definition of $\alpha$-denseness), and that the constant $C$ does
not depend on the sequences $\mathbf{H}$ or $\mathbf{q}$. That
Theorem~\ref{thm:rr:mathias2} is valid, however, follows easily from the proof
of~\cite[Theorem~3.3]{Schacht:KLR}. This can be checked as follows. It is
clearly stated in the proof of~\cite[Theorem~3.3]{Schacht:KLR} that the
requirement of $(\alpha,\eps,\zeta)$-denseness is necessary only once, namely
for the base case of the induction performed there, with the value
$\eps=8^{-r(r-1)/2}\delta$ given above. The values of the various constants are
also explicitly stated in the proof. In particular, the value of $C$ does indeed
depend only upon $r$, $K$, $\delta$ and $\zeta$ as claimed.

\smallskip

To prove the $1$-statement of Theorem~\ref{thm:TurForGnp}, we need to further
modify the setting from~\cite{Schacht:KLR}: we do not have a sequence of fixed
hypergraphs, but instead a sequence of random objects $\cR^{(r)}(n,p_n)$. We
describe how to modify the above definitions appropriately, and explain why the transfer result we require, Corollary~\ref{cor:rr}, follows from Theorem~\ref{thm:rr:mathias2}.

\begin{definition}[$(\alpha,\eps,\zeta)$-dense for random hypergraphs] 
  Let $\mathbf{p}=(p_n)_{n\in\mathbb{N}}$ be a sequence of probabilities, and
  let $\alpha,\eps,\zeta\geq 0$ be constants. We say the random hypergraph
  $\cR^{(r)}(n,p_n)$ is \emph{a.a.s.\ $(\alpha,\eps,\zeta)$-dense} if a.a.s.\
  for $\cR_n=\cR^{(r)}(n,p_n)$, the following is true. For every $n$-vertex
  graph $G$ on $[n]$ with at least $(\alpha+\eps)\binom{n}{2}$ edges, the number
  of copies of $K_r$ in $G$ induced by hyperedges of $\cR_n$ is at least $\zeta
  e(\cR_n)$.
\end{definition}
Next, we modify the definition of boundedness.

\begin{definition}[$(K,\mathbf{q})$-bounded for random hypergraphs]
  Let $\mathbf{p}=(p_n)_{n\in\mathbb{N}}$ and
  $\mathbf{q}=(q_n)_{n\in\mathbb{N}}$ be sequences of probabilities and $K\geq 1$ be a constant. We say that the random hypergraph
  $\cR^{(r)}(n,p_n)$ is \emph{a.a.s.\ $(K,\mathbf{q})$-bounded} if the following
  holds a.a.s.\ for $\cR_n=\cR^{(r)}(n,p_n)$.
  For each $i\in[\binom{r}{2}-1]$ and $\tilde q\geq q_n$, we have
  \[\mu_i(\cR_n,\tilde q)\leq K\tilde q^{2i}\cdot\frac{e(\cR_n)^2}{n^2}\,.\]
\end{definition}

Using these definitions we obtain the following transference result using
random hypergraphs as a corollary
to Theorem~\ref{thm:rr:mathias2}.

\begin{corollary}\label{cor:rr} 
  Given $r\geq 3$, $K\geq 1$, $\delta>0$, $\zeta>0$ and
  $(\omega_n)_{n\in\NATS}$ with $\omega_n\to\infty$ as $n\to\infty$, let
  $\eps:=\delta/8^{\binom{r}{2}}$. There exists $C>1$ such that the following is
  true. Let $\mathbf{p}=(p_n)_{n\in\mathbb{N}}$ be a sequence of probabilities
  such that $\cR^{(r)}(n,p_n)$ is a.a.s.\
  $\big(\frac{r-2}{r-1},\eps,\zeta\big)$-dense. Let
  $\mathbf{q}=(q_n)_{n\in\mathbb{N}}$ be a sequence of probabilities such
  that $Cq_n<1/\omega_n$, such that for every integer~$L$, a.a.s.\
  $q_n^{r(r-1)/2}\cdot e\big(\cR^{(r)}(n,p_n)\big)>L$, and such that
  $\cR^{(r)}(n,p_n)$ is a.a.s.\ $(K,\mathbf{q})$-bounded.
  Then for $G=\Gnq[Cq_n]$ and $\cR_n=\cR^{(r)}(n,p_n)$ a.a.s.\ $\cR_n$ is
  $\delta$-Tur\'annical for $G$.
\end{corollary}
\begin{proof} 
  Given $r\geq 3$, $K\geq 1$, $\delta>0$, $\zeta>0$ and $(\omega_n)_{n\in\NATS}$
  with $\omega_n\to\infty$ as $n\to\infty$, let $C$ be the constant returned by
  Theorem~\ref{thm:rr:mathias2}. Let $\mathbf{p}$ and $\mathbf{q}$
  be sequences of probabilities satisfying the conditions of the corollary.
  
  We define a property $\mathcal A_n$ of $r$-uniform hypergraphs as follows. An
  $n$-vertex hypergraph $\cH_n$ has property $\mathcal A_n$ if for all $n$-vertex
      graphs~$H$ with $V(H)=V(\cH_n)$ and $e(H)\ge
      \big(\frac{r-2}{r-1}+\eps\big)\binom{n}{2}$ the number of copies of $K_r$ in~$H$ induced by hyperedges of $\cH_n$
      is at least $\zeta e(\cH_n)$.
  
  We claim that there is a monotone function~$\nu(n)$ tending to zero as~$n$
  tends to infinity with the following properties.
  \begin{enumerate}[label=\abc]
    \item\label{item:rr:1} Let~$P_1(n)$ be the probability that 
      $\cR_n=\cR^{(r)}(n,p_n)$ has the property~$\mathcal A_n$. Then $P_1(n)\ge
      1-\nu(n)$.
   \item\label{item:rr:2}  There is a function $L(n)$ tending to
      infinity such that the probability~$P_2(n)$ that for
      $\cR_n=\cR^{(r)}(n,p_n)$
      \begin{equation}\label{eq:rr:2}
        q_n^{r(r-1)/2}\cdot e\big(\cR_n\big)>L(n)
      \end{equation}
      is at least $1-\nu(n)$.
    \item\label{item:rr:3} The probability $P_3(n)$ that, for
    $\cR_n=\cR^{(r)}(n,p_n)$, we have for each $i\in[\binom{r}{2}-1]$ and
    $\tilde q\geq q_n$
      \begin{equation}\label{eq:rr:3}
        \mu_i(\cR_n,\tilde q)\leq K\tilde q^{2i}\cdot\frac{e(\cR_n)^2}{n^2}\,,
      \end{equation}
	  is at least $1-\nu(n)$.
    \end{enumerate}
  Items~\ref{item:rr:1} and~\ref{item:rr:3} are immediate from the definitions
  of $(\frac{r-2}{r-1},\eps,\zeta)$-denseness and $(K,\mathbf{q})$-boundedness,
  respectively. Item~\ref{item:rr:2} is immediate from the
  fact that for each $L$, a.a.s.\ $q_n^{r(r-1)/2}\cdot e\big(\cR_n\big)>L$ holds.

  Let $n_0$ be such that $\nu(n_0)<\tfrac{1}{3}$. We fix a sequence
  $\mathbf{R}=(\cR_n)_{n\in\mathbb{N}}$ of hypergraphs in the following way.
  For each~$n\ge n_0$, consider the set of all $n$-vertex hypergraphs satisfying
  Property~$\mathcal A_n$,~\eqref{eq:rr:2},
  and~\eqref{eq:rr:3}. This set is non-empty by choice of~$n_0$. Now let~$\cR_n$
  be the element of this set which maximises the probability $P_4(n)$ that
  the random graph $G=\Gnq[Cq_n]$ possesses a subgraph with at least
  $(\frac{r-2}{r-1}+\delta)\cdot e(G)$ edges which is undetected by $\cR_n$. For
  $n<n_0$ let~$\cR_n$ be an arbitrary $n$-vertex hypergraph.

  We deduce from Property~$\mathcal A_n$ that~$\mathbf{R}$ is
  $\big(\frac{r-2}{r-1},\eps,\zeta\big)$-dense (in the sense of
  Definition~\ref{def:dense}), from~\eqref{eq:rr:3} that~$\mathbf{R}$ is
  $(K,\mathbf{q})$-bounded (in the sense of
  Definition~\ref{def:bounded}), and from~\eqref{eq:rr:2} that~$\mathbf{R}$ satisfies
  $q_n^{r(r-1)/2}\cdot e(\cR_n)\rightarrow\infty$.
  It follows that we can apply Theorem~\ref{thm:rr:mathias2} to $\mathbf{R}$,
  which implies that the probability $P_4(n)$ tends to zero as $n$ tends to
  infinity.  Consequently, with probability at least
  $1-\big((1-P_1(n))+(1-P_2(n))+(1-P_3(n))\big)-P_4(n)\ge
  1-3\nu(n)-P_4(n)=1-o(1)$, the random hypergraph $\cR^{(r)}(n,p_n)$ detects every subgraph of $G=\Gnq[Cq_n]$ with at
  least $(\frac{r-2}{r-1}+\delta)\cdot e(G)$ edges. Hence $\cR^{(r)}(n,p_n)$ is
  a.a.s.\ $\delta$-Tur\'annical for $\Gnq[Cq_n]$.
\end{proof}

To prove the $1$-statement of Theorem~\ref{thm:TurForGnp} it now suffices to
check that the conditions of Theorem~\ref{thm:TurForGnp} guarantee that
$\cR^{(r)}(n,p)$ satisfies the conditions of Corollary~\ref{cor:rr}. We will
make use of the Chernoff bound for a binomial random variable $X$ (see, e.g.,
\cite[Theorem~2.1]{JaLuRu:Book})
\begin{equation}\label{eq:chernoff:LB}
  \Prob\big(X\ge(1+\gamma)\Exp X\big)\le\exp(-\gamma^2\Exp X/3)\,,
  \qquad\text{for $\gamma\le1/2$}\,.
\end{equation}

\medskip
The last tool we shall need for our proof of Theorem~\ref{thm:TurForGnp} is a
counterpart of Theorem~\ref{thm:TuranDensity} for random graphs due to
Kohayakawa, R\"odl and Schacht.
\begin{theorem}[Kohayakawa, R\"odl
\& Schacht~\cite{KoRoSch:TuranOld}]\label{thm:KoSchRoOLDTURAN} Given any
$r\in\mathbb N$ and $\eps>0$, there exists $\delta>0$ such that for any
sequence of probabilities $(q_n)_{n\in\mathbb N}$ with $\liminf_n q_n>0$
the following is a.a.s.\ true for the random graph $G=\Gnq[q_n]$. If
$G'\subset G$ is a graph with at least  $(1+\eps)\frac{r-2}{r-1}e(G)$
edges, then there are at least~$\delta q_n^{r(r-1)/2}n^r$ copies of $K_r$  in $G'$.
\end{theorem}
Kohayakawa, R\"odl and Schacht prove their result for a wider range of
probabilities, allowing $q_n$'s to decrease roughly at the speed $n^{-\frac1{r-1}}$.
However we do not need this stronger result. Actually, in our setting when
$\liminf_n q_n>0$, Theorem~\ref{thm:KoSchRoOLDTURAN} has a relatively simple proof using Szemer\'edi's Regularity Lemma. Let us remark that Theorem~\ref{thm:KoSchRoOLDTURAN} was one of the early contributions to the Kohayakawa-{\L}uczak-R\"odl conjecture. 

\begin{proof}[Proof of Theorem~\ref{thm:TurForGnp}] 
  Given $r$ and $\eps\in(0,1/(r-2))$, set $\delta':=\eps$ and
  $\eps':=\delta'/8^{\binom{r}{2}}$. Let $\zeta>0$ be the constant provided by
  Theorem~\ref{thm:TuranDensity} for~$r$ and~$\eps'$. Now
  set
  \begin{equation}\label{eq:Kprime}
    K':=r^{2r+5} 2^{r^2+3}
  \end{equation}
  and let $C'$ be the constant returned by
  Corollary~\ref{cor:rr} for input~$r$, $K'$, $\delta'$ and $\zeta$. 
  Let $\deltaH$ be given by Theorem~\ref{thm:KoSchRoOLDTURAN} for input parameters $\eps$ and $r$. Set
 \begin{equation}\label{eq:TurForGnp:c}
    c:=\tfrac{1}{16}\left(\tfrac1{r-1}-\eps\tfrac{r-2}{r-1}\right)
    \qquad\text{and}\qquad
    C:= \max\left\{ \tfrac{8}{\zeta}, C'^{(r+1)(r-2)/2},\tfrac{2}{\deltaH} \right\}
    \,. 
  \end{equation}
  The constants $c$ and $C$ from~\eqref{eq:TurForGnp:c} define the
  thresholds for the 0-statement and 1-statement of Theorem~\ref{thm:TurForGnp}.
  Let $\mathbf{p}=(p_n)_{n\in\NATS}$ and $\mathbf{q}=(q_n)_{n\in\NATS}$ be
  given. We let $\mathcal{T}_n$ denote the event that $\cR^{(r)}(n,p_n)$ is
  $\eps$-Tur\'annical for $\Gnq[q_n]$.

  \smallskip

  First we prove the \emph{$0$-statement}. 
  Since adding hyperedges to a sequence of hypergraphs does not
  destroy their property of being a.a.s.\ $\eps$-Tur\'annical for $\Gnq[q_n]$,
  we can assume that
  \begin{equation}\label{eq:TurForGnp:0}
    p_n=c\big(nq_n^{(r+1)/2}\big)^{2-r} 
    \quad\text{and hence}\quad
    q_n=c' \big(np_n^{1/(r-2)}\big)^{-2/(r+1)}
    \,,
  \end{equation}
  where $c':=c^{2/((r+1)(r-2))}$. In particular, since $1\ge p_n$, we have that 
  \begin{equation}\label{eq:plar}
    q_n
    \ge c' n^{-2/(r+1)}
 \end{equation}

  Recall that we are dealing with two random objects, the random graph $\Gnq[q_n]$ and
  the random hypergraph $\cR^{(r)}(n,p_n)$.
  In the
  following argumentation we shall first perform the random experiment for
  $\Gnq[q_n]$ and then the one for $\cR^{(r)}(n,p_n)$.

  Let us first expose the graph $\Gnq[q_n]$. The Chernoff
  bound~\eqref{eq:chernoff} implies that the probability that $\Gnq[q_n]$
  has less than $q_nn^2/4$ edges tends to zero.
  Moreover, the random variable~$X$ counting copies of $K_r$ in $\Gnq[q_n]$ has
  expectation $\binom{n}{r}q_n^{r(r-1)/2}$ and variance
  $\bigO\big(n^rq_n^{r(r-1)/2}\big)$ (see for example Lemma~3.5
  of~\cite{JaLuRu:Book}). Hence, applying Chebyshev's inequality and observing
  that $n^rq_n^{r(r-1)/2}\to\infty$ by~\eqref{eq:plar}, we obtain that
 $\Prob\big[X\ge 2\tbinom{n}{r}q_n^{r(r-1)/2}\big]=o(1)$.

  Since a.a.s. $\Gnq[q_n]$ has at least $q_nn^2/4$ edges and 
  \begin{equation}\label{eq:Ptail1}
    X<2\tbinom{n}{r}q_n^{r(r-1)/2}\,,
  \end{equation}
  from now we assume these two events
  occur. We next expose the hypergraph $\cR^{(r)}(n,p_n)$.
  Let~$Y$ be the random variable counting the hyperedges of
  $\cR^{(r)}(n,p_n)$ which induce copies of $K_r$ in~$G=\Gnq[q_n]$.  
  Observe that~$Y$ has distribution $\Bin\big(X,p_n\big)$.
  From the Chernoff bound~\eqref{eq:chernoff:LB} and from~\eqref{eq:Ptail1} we
  infer that a.a.s.\ $Y$ does not exceed $4\binom{n}{r}q_n^{r(r-1)/2}p_n$.
  Because
  \begin{equation}\label{eq:TurForGnp:EY}
    \binom{n}{r}q_n^{\binom{r}{2}}p_n=
    \binom{n}{r}
    q_n^{(r-2)(r+1)/2}p_nq_n\eqByRef{eq:TurForGnp:0}\binom{n}{r}cn^{2-r}q_n
    \,,
  \end{equation}
  we thus a.a.s.\ have
  \begin{equation*}\label{eq:cliq}
    \begin{split}
     Y&\le 4\tbinom{n}{r}q_n^{\binom{r}{2}}p_n 
     \eqByRef{eq:TurForGnp:EY}4\tbinom{n}{r}cn^{2-r}q_n
     \le 4c q_n n^2 \\
     &\eqByRef{eq:TurForGnp:c} \left(\frac1{r-1}-\eps\frac{r-2}{r-1}\right)
     \frac{q_nn^2}{4}
     \le \left(\frac1{r-1}-\eps\frac{r-2}{r-1}\right)e(G)\,.
   \end{split}
 \end{equation*}
 Hence, a.a.s.\ $\cR_n=\cR^{(r)}(n,p_n)$ does not detect some subgraph~$G'$
 of~$G$ which is obtained by deleting at most
 $(\frac1{r-1}-\eps\frac{r-2}{r-1})e(G)$ edges from~$G$. In particular,
 $e(G')\ge(1+\eps)\frac{r-2}{r-1}e(G)$,
 which finishes the proof of the $0$-statement.

  \smallskip

  We now turn to the \emph{$1$-statement}.
  Again, by monotonicity, we can assume that
  \begin{equation}\label{eq:TurForGnp:1}
    p_n=C\big(nq_n^{(r+1)/2}\big)^{2-r} 
    \quad\text{and hence}\quad
    q_n=C_q \big(np_n^{1/(r-2)}\big)^{-2/(r+1)}
    \,,
  \end{equation}
  where $C_q:=C^{2/((r+1)(r-2))}\geByRef{eq:TurForGnp:c} C'$.
  Since $p_n\le 1$ and $q_n\le 1$ we have that 
  \begin{equation}\label{eq:tI}
    q_n \ge C_q n^{-2/(r+1)} \quad \mbox{ and }\quad p_n \ge Cn^{2-r}\;.
  \end{equation}
  
  We can assume (by taking subsequences if it is necessary) that either $\liminf_n q_n>0$, or $q_n=o(1)$. In the former case we mimic our proof of Theorem~\ref{thm:random:approx} while in the latter case we apply Corollary~\ref{cor:rr}.
  
  \medskip
  Let us first prove the 1-statement when $\liminf_n q_n>0$. We repeat the
  proof strategy of the 1-statement of
  Theorem~\ref{thm:random:approx}. Suppose that $G'$ is an arbitrary graph
  on the vertex set $[n]$ with at least $\deltaH q_n^{r(r-1)/2}n^r$ copies
  of $K_r$. The probability that $G'$ is not detected by $\cR^{(r)}(n,p_n)$
  is at most \[(1-p_n)^{\deltaH q_n^{r(r-1)/2}n^r}\;.\] Suppose now that a
  random graph $G=\Gnq[q_n]$ is given. We can assume that $G$ has at most
  $q_nn^2$ edges as this property is a.a.s.\ satisfied. Consequently, $G$ contains at most $2^{q_nn^2}$ subgraphs $G'$ on the same vertex set. By Theorem~\ref{thm:KoSchRoOLDTURAN} we a.a.s.\ have that each such subgraph with at least $(1+\eps)\frac{r-2}{r-1}e(G)$ edges contains at least $\deltaH q_n^{r(r-1)/2}n^r$ copies of $K_r$. Therefore, the union bound over all such graphs $G'$ gives that
  \begin{multline*}
    \Prob\left[\cR^{(r)}(n,p_n)\mbox{ is not $\eps$-Tur\'annical for }
      \Gnq[q_n]\right]
    \le 2^{q_nn^2}\cdot (1-p_n)^{\deltaH q_n^{\binom{r}2}n^r}\\
    \le \exp\left(q_n n^2-p_n \deltaH q_n^{\binom{r}2}n^r \right)
    \eqByRef{eq:TurForGnp:1}\exp\left(q_nn^2-Cn^2q_n\deltaH\right)
    \ByRef{eq:TurForGnp:c}{\rightarrow} 0\;,
  \end{multline*}
  and the statement follows in this case.

  \medskip
  Let us now focus on the 1-statement in the case $q_n=o(1)$. The claim will follow from
  Corollary~\ref{cor:rr} (with parameters $r$, $K'$, $\delta'$, $\zeta$, $C'$)
  applied to the sequences of
  probabilities $\mathbf{p}$ and $\mathbf{q}'=(q'_n)_{n\in\NATS}:=\mathbf{q}/C'$,
  together with the following claim.

  \begin{AuxiliaryCl}
    \label{cl:TurForGnp}
    We have that
    \begin{enumerate}[label=\abc]
      \item\label{cl:TurForGnp:edgecount} for every~$L$ a.a.s.\
        $(q'_n)^{r(r-1)/2}\cdot e\big(\cR^{(r)}(n,p_n)\big)>L$, 
      \item\label{cl:TurForGnp:dense} $\cR^{(r)}(n,p_n)$ is a.a.s.\
        $\big(\frac{r-1}{r-2},\eps',\zeta\big)$-dense, and
      \item\label{cl:TurForGnp:bounded} $\cR^{(r)}(n,p_n)$ is a.a.s.\
        $(K',\mathbf{q}')$-bounded.
      \end{enumerate}
    \end{AuxiliaryCl}
    \begin{clproof}[Proof of Claim~\ref{cl:TurForGnp}]
      We first verify~\ref{cl:TurForGnp:edgecount}. We have
      \[\Exp\Big(e\big(\cR^{(r)}(n,p_n)\big)\Big)=p_n\binom{n}{r}\;,\]
      which tends to infinity by~\eqref{eq:tI}. Consequently, the Chernoff
      bound~\eqref{eq:chernoff} guarantees that a.a.s.\ $\cR^{(r)}(n,p_n)$ has
      at least $p_n\binom{n}{r}/2$ hyperedges. Now we have
      \[\frac{(q'_n)^{\binom{r}{2}}p_n\binom{n}{r}}{2}
      \eqByRef{eq:TurForGnp:1}\frac{(q'_n)^{\frac{r^2-r}{2}}Cn^{2-r}
      q_n^{(r+1)(2-r)/2}\binom{n}{r}}{2}=\Omega\left(q_n
      n^{2-r}\binom{n}{r}\right)\,, \]
and by~\eqref{eq:tI} this tends to infinity.
      
      \smallskip

      Now we verify~\ref{cl:TurForGnp:dense}. Given an $n$-vertex graph $H$ with
      $e(H)\geq\big(\frac{r-2}{r-1}+\eps'\big)\binom{n}{2}$, by
      Theorem~\ref{thm:TuranDensity}, $H$ contains at least $\zeta n^r$ copies of
      $K_r$. It follows that the expected number of hyperedges of
      $\cR_n=\cR^{(r)}(n,p_n)$ which induce copies of $K_r$ in $H$ is at least
      $\zeta n^r p_n$. By the Chernoff bound~\eqref{eq:chernoff}, the
      probability that less than $\zeta n^r p_n/2$ copies of $K_r$ in $H$ are
      induced by hyperedges of $\cR_n$ is at most \[\exp\left(-\frac{\zeta n^r p_n}{8}\right)
      \leByRef{eq:tI}\exp\left(-\frac{C\zeta n^2}{8}\right) \eqByRef{eq:TurForGnp:c}
      o(2^{-n^2})\,.\]
      Applying the union bound (on at most $2^{\binom{n}{2}}$ graphs $H$) we
      conclude that the probability that there exists any $n$-vertex graph~$H$ with at least
      $\big(\frac{r-1}{r-2}+\eps'\big)\binom{n}{2}$ edges and less than
      $3\zeta{\binom{n}{r}}p_n/2\le \zeta n^r p_n/2$ copies of $K_r$ on
      hyperedges of $\cR_n$ tends to zero as $n$ tends to infinity. Furthermore, applying
      the Chernoff bound~\eqref{eq:chernoff:LB} in conjunction with~\eqref{eq:tI}, the probability that $\cR^{(r)}(n,p)$ has
      more than $3p_n\binom{n}{r}/2$ hyperedges tends to zero as $n$ tends to
      infinity. It follows that for~$\cR_n$ a.a.s.\ every $n$-vertex graph $H$ with more than
      $\big(\frac{r-2}{r-1}+\eps'\big)\binom{n}{2}$ edges has at least $\zeta
      e(\cR_n)$ copies of $K_r$ on hyperedges of $\cR_n$. Therefore,
      $\cR^{(r)}(n,p_n)$ is a.a.s.\
      $\big(\frac{r-2}{r-1},\eps',\zeta\big)$-dense.

      \smallskip

     Now we prove~\ref{cl:TurForGnp:bounded}. We need to show that $\cR_n=\cR^
     {(r)}(n,p_n)$ a.a.s.\ has the property that for each $1\leq
    i\leq\binom{r}{2}-1$ and each $\tilde q\geq q'_n$, we have
      \begin{equation}\label{eq:muCond}
        \mu_i(\cR_n,\tilde q)\leq
        K'\tilde q^{2i}\frac{e(\cR_n)^2}{n^2}\,.
      \end{equation}
      We will show that~\eqref{eq:muCond} holds for
      all $1\leq i\leq\binom{r}{2}-1$ and $\tilde q\geq q'_n$ provided that
      $\cR_n$ obeys a simple bound (inequality~\eqref{eq:ToVer} below); this bound will
      turns out to hold a.a.s.\ for our random hypergraph.
      
      Given a hypergraph $\cR_n$ and two distinct vertices $u$ and $v$, let
      $F_1$ and $F_2$ be two hyperedges containing $u$ and $v$ and
      intersecting in a set $A$ of $j$ vertices.  Then the probability~$P_{i,j}$
      that both $F_1$ and $F_2$ contain at least $i$ edges of the random graph $G=\Gnq[\tilde q]$, not counting
      $uv$, can be bounded as follows. We use the random variables
      $X_A:=|E(G[A])\setminus uv|$, $X_{F_1}:=e(G[F_1\setminus
      A])+e(G[F_1\setminus A,A])$, and $X_{F_2}:=e(G[F_2\setminus A])+e(G[F_2\setminus A,A])$. Then
      \begin{equation}\begin{split}
          P_{i,j}
          & \le \sum_{k=0}^{\binom{j}{2}-1} 
            \Prob(X_A=k)\Prob(X_{F_1}\geq i-k)\Prob(X_{F_2}\geq i-k) \\
          & \le \sum_{k=0}^{\binom{j}{2}-1}
            \binom{\binom{j}{2}-1}{k}\tilde q^k
            \bigg(\binom{\binom{r}{2}-\binom{j}{2}}{i-k}\tilde q^{i-k}\bigg)^2  \\
          & \le 2^{\binom{j}{2}-1+2\left(\binom{r}{2}-\binom{j}{2}\right)} 
            \sum_{k=0}^{\binom{j}{2}-1}\tilde q^{2i-k} \le j^2 2^{r^2} \cdot
            \tilde q^{2i+1-\binom{j}{2}} \,.\label{eq:Pij}
      \end{split}\end{equation}
      Let $N(j)$ count the number of pairs of hyperedges in $\cR_n$
      intersecting in exactly $j$ vertices. Then we have
      \begin{equation*}\begin{split}
        \mu_i(\cR_n,\tilde q)
        & =\mathbb{E}\left[\sum_{\substack{u,v\\u\neq
        v}}\deg_i^2(u,v,\Gnq[\tilde q])\right] =\sum_{\substack{u,v\\u\neq v}}\:
        \sum_{\substack{F_1\in\cE(\cR_n)\\F_1\ni
        u,v}}\:\sum_{\substack{F_2\in\cE(\cR_n)\\F_2\ni u,v}} P_{i,|F_1\cap
        F_2|} \\ & =\sum_{j=2}^{r}N(j)j(j-1)P_{i,j} \leByRef{eq:Pij} r^4 2^{r^2}\sum_{j=2}^{r} N(j)\tilde q^{2i+1-\binom{j}{2}}\,.
      \end{split}\end{equation*} 
      It follows that $\cR_n$ satisfies~\eqref{eq:muCond} if we have, for
      each $2\leq j\leq r$ and $\tilde q\geq q'_n$,
      \begin{equation}\label{eq:ToVer}
        r^5 2^{r^2}\cdot N(j)\cdot \tilde q^{1-\binom{j}{2}}
        \le K' \frac{e(\cR_n)^2}{n^2}\,.
      \end{equation}      
      Since $j\geq 2$ we have $1-\binom{j}{2}\leq 0$. Therefore, the left-hand
      side of~\eqref{eq:ToVer} is non-increasing in~$\tilde q$. The right-hand
      side of~\eqref{eq:ToVer} does not depend upon $\tilde q$. It follows that we need only verify that a.a.s.\ $\cR_n=\cR^{(r)}(n,p_n)$ satisfies~\eqref{eq:ToVer}  for each $2\leq j\leq r$, with $\tilde q=q'_n$. We have that a.a.s.\
      $e(\cR^{(r)}(n,p_n))\geq p_n\binom{n}{r}/2\ge p_n n^r/(2r^r)$, by the Chernoff
      bound~\eqref{eq:chernoff}. So it is enough to show that a.a.s.\ for
      each $2\leq j\leq r$ we have
      \begin{equation}\label{eq:Nj}
        N(j)\le \frac{K'}{r^5 2^{r^2}} (q'_n)^{\frac{(j-2)(j+1)}{2}} 
          \frac{p_n^2n^{2r-2}}{4 r^{2r}}
        \eqByRef{eq:Kprime} 2 (q'_n)^{\frac{(j-2)(j+1)}{2}} p_n^2n^{2r-2}
      \,.
      \end{equation}

      To show that~\eqref{eq:Nj} holds, we first consider the case $j=r$.
      Observe that $N(r)$ is simply
      the number of hyperedges in $\cR^{(r)}(n,p_n)$, and is therefore (by the
      Chernoff bound~\eqref{eq:chernoff:LB}) a.a.s.\ at most
      $2p_n\binom{n}{r}\le 2p_n n^r$.  Substituting
      $q'_n\ge \big(np_n^{1/(r-2)}\big)^{-2/(r+1)}$ into the
      right-hand side of~\eqref{eq:Nj} (for $j=r$), we have
      \[2 (q'_n)^{\frac{(r-2)(r+1)}{2}}p_n^2n^{2r-2}
      \ge 2\Big(np_n^{\frac{1}{r-2}}\Big)^{2-r}p_n^2n^{2r-2}
      =2 p_n n^r\;.
      \]
      Therefore~\eqref{eq:Nj} holds for $j=r$.

      Suppose now that $2\leq j\leq r-1$. Then we have
      \[\Exp(N(j))=\binom{n}{r}\binom{r}{j}\binom{n-r}{r-j}p_n^2
      =\bigO(n^{2r-j}p_n^2)\,.\]
      We have by~\eqref{eq:tI} that
      $q'_n=\Omega\big(n^{-\frac{2}{r+1}}\big)=\omega\big(n^{-\frac{2}{j+1}}\big)$
      for each $2\leq j\leq r-1$. Consequently,
      \[\Exp(N(j))=\bigO(n^{2r-j}p_n^2)=\bigO(n^{2-j}p_n^2n^{2r-2})=
      o\Big((q'_n)^{\frac{(j-2)(j+1)}{2}}p_n^2n^{2r-2}\Big)\,.\]
      By Markov's inequality,~\eqref{eq:Nj} holds a.a.s.\ for
      every $2\leq j\leq r-1$. This completes the proof that
      $\cR^{(r)}(n,p_n)$ is a.a.s.\ $(K',\mathbf{q'})$-bounded.
    \end{clproof}

    It follows that a.a.s.\ $\cR^{(r)}(n,p_n)$ satisfies the conditions to
    apply Corollary~\ref{cor:rr}, that is,  a.a.s.\ $\cR^{(r)}(n,p_n)$ is
    $\eps$-Tur\'annical for $\Gnq[q_n]$.
\end{proof}

\section{Sharp thresholds} \label{sec:SharpThresholds}

In this section we use Friedgut's~\cite{Fried:Hunting} condition for
sharp thresholds to prove that the threshold we obtained in
Theorem~\ref{thm:random:exact} is sharp.  For a background on threshold
phenomena we refer the reader to~\cite{Fried:Hunting}.  We show the
following result.

\begin{theorem}
\label{thm:sharp:exact}
  For every integer $r\ge 3$ there
  are $c,C>0$ and a sequence of numbers
  $(c_n\in(c,C))_{n\in\mathbb N}$ such that for every
  $\gamma>0$ we have 
  \begin{align*}
    \lim_{n\to\infty}
    \Prob\big( \text{$\cR^{(r)}\big(n,(c_n-\gamma) n^{3-r}\big)$
    is Tur\'annical}\,\big) &=0\;\quad\mbox{and, }\\
    \lim_{n\to\infty}
    \Prob\big( \text{$\cR^{(r)}\big(n,(c_n+\gamma) n^{3-r}\big)$
    is Tur\'annical}\,\big) &=1\;.
  \end{align*}
\end{theorem}
As usual it is reasonable to conjecture that the sequence
$(c_n)$ in this theorem converges, and as usual in the
field we are not able to prove this.

\smallskip

Before we can state Friedgut's result we need to introduce some notation.
Given two hypergraphs $\cG$ and $\cM$ with $|V(\cG)|\ge |V(\cM)|$ we write
$\cG\cup \cM^*$ for the random hypergraph obtained from the following
random experiment. Let~$\phi$ be a (uniformly chosen) random injection from
$V(\cM)$ to $V(\cG)$ and for each hyperedge~$F$ of~$\cM$ add the
hyperedge~$\phi(F)$ to~$\cG$ (without creating multiple hyperedges).
A family of $r$-uniform hypergraphs is called a \emph{hypergraph property}
if it is closed under isomorphism and under adding hyperedges.

Friedgut formulates his result for graphs. Here, we
use the corresponding hypergraph result, specialised to our situation; see
also~\cite{Friedgut99:Sharp} for a discussion of this result and for
extensions to other combinatorial structures.

\begin{theorem}[Friedgut~{\cite[Theorem~2.4]{Fried:Hunting}}]\label{thm:FrSpec}
  Suppose that Theorem~\ref{thm:sharp:exact} does not hold for some
  $r\ge 3$. Then there exists a sequence $p=p_n$, $\tau>0$, a fixed $r$-uniform
  hypergraph $\cM$ with
  \begin{equation}\label{eq:MBalanced}
    \Prob\big(\cM\subset \cR^{(r)}(n,p)\big)>\tau\;,
  \end{equation}
  and $\alpha>0$ with
  \begin{equation}\label{eq:ProbBounded}
    \alpha<\Prob\big(\cR^{(r)}(n,p)\mbox{ is
      Tur\'annical}\big)<1-3\alpha\;,
  \end{equation}
  and a constant $\eps>0$ such that, for every hypergraph property $\mathcal
  P$ which satisfies that $\cR^{(r)}(n,p)$ is a.a.s.\ in $\mathcal P$, the following holds.
  There exists an
  infinite set $Z\subset\mathbb N$ and for each $n\in Z$ a hypergraph
  $\cG_n\in\mathcal P$ such that
  \begin{align}
    \label{eq:boostM}
    \Prob\big(\cG_n\cup \cM^*\mbox{ is
      Tur\'annical}\big)&>1-\alpha\;,\\
    \label{eq:BoostGnp}
    \Prob\big(\cG_n\cup \cR^{(r)}(n,\eps p)\mbox{ is
      Tur\'annical}\big)&<1-2\alpha\;.
  \end{align}
\end{theorem}

With this result at hand, we can now give a proof of
Theorem~\ref{thm:sharp:exact}. It turns out that we do not need to utilise
Theorem~\ref{thm:FrSpec} in its full strength; in
particular we shall not use assertion~\eqref{eq:MBalanced}.

\begin{proof}[Proof of Theorem~\ref{thm:sharp:exact}]
Suppose that Theorem~\ref{thm:sharp:exact} does not hold for
some $r\ge 3$. Let $p_n$, the $r$-uniform hypergraph $\cM$, and
$\alpha>0$ be given by Theorem~\ref{thm:FrSpec}. In particular, by~\eqref{eq:ProbBounded} we have that $\alpha<1/4$. It follows
from~\eqref{eq:ProbBounded} and from
Theorem~\ref{thm:random:exact} that 
\[cn^{3-r}\le p\le Cn^{3-r}\;,\]
for some absolute constants $c,C>0$.
Let $\beta:=\frac{1}{2e(\cM)}$ and let $\mathcal P$ be the family
of $n$-vertex hypergraphs which detect every $n$-vertex
graph $F$ with at least $\beta \binom{n}{r}$ $r$-cliques. It
follows from the proof of Theorem~\ref{thm:random:approx} 
that a.a.s.\ $\cR^{(r)}(n,p)\in \mathcal P$.

Let now $Z\subset \mathbb N$ 
and $(\cG_n)_{n\in Z}$
be given by
Theorem~\ref{thm:FrSpec}. We will derive a contradiction using
just a single hypergraph $\cG_n$, $n\in Z$. Indeed,
from~\eqref{eq:BoostGnp} we see that $\cG_n$ itself cannot be
Tur\'annical. Let $W$ be a graph which witnesses this, i.e., $W$
is an $n$-vertex graph with more than $\tur{r}{n}$ edges which
is not detected by $\cG_n$. By the definition of $\mathcal P$ and since
$\cG_n\in\mathcal P$, the graph $W$
contains less than $\beta\binom{n}{r}$ $r$-cliques. If $\cG_n\cup
\cM^*$ is Tur\'annical then at least one hyperedge of $\cM$ must be
placed on an $r$-clique of $W$. 
Therefore we have
\[\Prob\big(\cG_n\cup \cM^*\mbox{ is Tur\'annical}\big)\le
e(\cM)\beta<\frac12\;,\]
which contradicts~\eqref{eq:boostM}.
\end{proof}

\section{Random restrictions}
\label{sec:conc}

Traditional extremal combinatorics deals with questions in the
following framework. Given a combinatorial structure $\cS$ (such as the edge set
of the complete graph $K_n$, or the set $2^{[n]}$ of subsets of $[n]$) and a
monotone increasing parameter $f\colon 2^{\cS}\rightarrow\mathbb{N}$ (such as the
minimum degree of $H\subset K_n$, or the number of sets in the set family
$H\subset 2^{[n]}$), we ask:
\begin{quote}
What is the maximum possible value $f(H)$ for $H\subset\cS$ satisfying a set of
restrictions $\cR$?
\end{quote}
Often the restrictions $\mathcal R$ are simply \emph{all} substructures
of~$\cS$ of a certain type.
For example, in the setting of Tur\'an's theorem every $r$-tuple of
vertices forbids a clique; in that of Sperner's theorem~\cite{Sperner}, every
pair of sets $A\subset B\subseteq [n]$ is forbidden to be in the set family
$H\subset 2^{[n]}$.

In this framework there are two places where randomness may come into play.
Firstly, one could choose~$\cS$ to be a random structure (and thus~$H$ be a
substructure of a random structure). A famous example of this type of
randomness is the Kohayakawa-{\L}uczak-R\"odl conjecture concerning a
version of Tur\'an's theorem for random graphs
(see~\cite{KLR}) mentioned already in the introduction. 
Versions of the famous Erd\H{o}s-Ko-Rado theorem 
for random hypergraphs as studied by
Balogh, Bohman, and Mubayi~\cite{BalBohMub} form another example.

Secondly, the restriction
set can be relaxed to a random subset of all possible
restrictions~$\mathcal R$. This is exemplified in
Theorems~\ref{thm:random:approx} and~\ref{thm:random:exact} in the context
of Tur\'an's theorem. Moreover, the two types of randomness can be
combined, as shown in Theorem~\ref{thm:TurForGnp}.

Obviously, similar randomised versions can be formulated for many
other problems. Probably the closest one to the present paper would be a
variant of the Erd\H{o}s-Stone theorem about the extremal number of $H$-free
graphs with random restrictions. While the statement and the proof of
Theorem~\ref{thm:random:approx} translates \emph{mutatis mutandis} to that
setting when $\chi(H)\ge 3$, obtaining either a proof for $\chi(H)=2$ or an
analogue of Theorem~\ref{thm:random:exact} seem to be significantly
harder. 
We conclude by mentioning two additional problems which seem
interesting for further research.

\smallskip 

\noindent{\bf Ramsey theory.}
  Graph Ramsey theory deals with estimating the parameter $R(H)$, which is the
  smallest number $n$ such that any two-colouring of edges of the complete graph
  $K_n$ contains a monochromatic copy of $H$.
 
  In a randomised version of this problem of the first type mentioned
  above, we colour the edges of the random graph $\Gnq$ instead of $K_n$
  and search for a monochromatic copy of~$H$ in such a colouring.  The
  threshold for this problem was determined by R\"odl and
  Ruci{\'n}ski~\cite{RoRu:RamseyThreshold95} (see also Friedgut, R\"odl and
  Schacht~\cite{FrRoSch:Randomized} and Conlon and Gowers~\cite{ConGow:KLR}
  for some recent progress).
  
  Concerning the second approach for randomisation mentioned above, we
  suggest considering the following problem. Given~$n$ and a
  probability~$p$, let $\cR(n,p)$ be a set of copies of~$H$ in~$K_n$
  obtained by picking $H$-copies independently at random with
  probability~$p$ from the set of all copies of~$H$ in~$K_n$. What is the
  threshold $p=p_n$ such that a.a.s.\ $\cR=\cR(n,p)$ has the property that
  for every two-edge-colouring of $K_n$, there is a monochromatic copy of
  $H$ contained in $\cR$?

\smallskip 

\noindent{\bf VC-dimension.}
  The celebrated Sauer-Shelah Lemma~\cite{Sauer:VC,Shelah:VC} states that if
  $\mathcal A$ is a family of subsets of $[n]$ with $|\mathcal
  A|>\binom{n}{0}+\ldots+\binom{n}{k-1}$ then there is a set $X\subset [n]$ of
  size $k$ which is \emph{shattered} by $\mathcal A$, i.e., for every $Y\subset
  X$, there is $A\in\mathcal A$ such that $Y=X\cap A$.
  
  A randomised variant of this Lemma of the first type mentioned above
  would generate a random family $\mathcal X=\binom{[n]}{k}_p$ of $k$-sets
  in $[n]$, each $k$-set being present in this family independently
  with probability $p=p_n$. The question is then: How large must
  $|\mathcal A|$ be in order to guarantee a shattered $k$-set
  $X\in\mathcal X$?

  A randomised version of the second type, instead, would randomise
  the concept of a shattering in the Sauer-Shelah Lemma.  More precisely, a
  \emph{$p$-shattering} does not require every subset $Y\subset X$ to be
  represented as $X\cap A$ for some~$A\in \mathcal A$, but only for each
  $X\subset[n]$ of size $k$ a family of subsets~$Y$ which are selected
  randomly and independently from $2^X$ with probability~$p$.  The question
  then is: Given $0<c\le 1$, what is the threshold $p=p_n$ such that
  a.a.s.\ there exists a set family with
  $c\big(\binom{n}{0}+\ldots+\binom{n}{k-1}\big)$ members which does not
  even $p$-shatter any $k$-set in $[n]$?

\section*{Acknowledgement}
We thank Yoshiharu Kohayakawa for stimulating discussions, and an anonymous
referee for detailed comments.


\bibliographystyle{amsplain_yk} 
\bibliography{bibl}

\providecommand{\bysame}{\leavevmode\hbox to3em{\hrulefill}\thinspace}
\providecommand{\MR}{\relax\ifhmode\unskip\space\fi MR }
\providecommand{\MRhref}[2]{%
  \href{http://www.ams.org/mathscinet-getitem?mr=#1}{#2}
}
\providecommand{\href}[2]{#2}
\def\MR#1{\relax}
\begin{thebibliography}{10}

\bibitem{BalBohMub}
J.~Balogh, T.~Bohman, and D.~Mubayi, \emph{Erd{\H o}s-{K}o-{R}ado in random
  hypergraphs}, Combin. Probab. Comput. \textbf{18} (2009), no.~5, 629--646.

\bibitem{BriPanSte}
G.~Brightwell, K.~Panagiotou, and A.~Steger, \emph{On extremal subgraphs of
  random graphs}, SODA '07: Proceedings of the eighteenth annual ACM-SIAM
  symposium on Discrete algorithms (Philadelphia, PA, USA), Society for
  Industrial and Applied Mathematics, 2007, pp.~477--485.

\bibitem{ConGow:KLR}
D.~Conlon and T.~Gowers, \emph{Combinatorial theorems in sparse random sets},
  Preprint (arXiv:1011.4310).

\bibitem{ErdosStone1946}
P.~Erd\H{o}s and A.~H. Stone, \emph{On the structure of linear graphs},
  Bulletin of the American Mathematical Society \textbf{52} (1946), 1087--1091.

\bibitem{ErdSim:Supersaturated}
P.~Erd{\H{o}}s and M.~Simonovits, \emph{Supersaturated graphs and hypergraphs},
  Combinatorica \textbf{3} (1983), no.~2, 181--192.

\bibitem{FrRo:K3K486}
P.~Frankl and V.~R{\"o}dl, \emph{Large triangle-free subgraphs in graphs
  without {$K_4$}}, Graphs Combin. \textbf{2} (1986), no.~2, 135--144.

\bibitem{Friedgut99:Sharp}
E.~Friedgut, \emph{Sharp thresholds of graph properties, and the {$k$}-sat
  problem}, J. Amer. Math. Soc. \textbf{12} (1999), no.~4, 1017--1054, With an
  appendix by Jean Bourgain.

\bibitem{Fried:Hunting}
\bysame, \emph{Hunting for sharp thresholds}, Random Structures Algorithms
  \textbf{26} (2005), no.~1-2, 37--51.

\bibitem{FrRoSch:Randomized}
E.~Friedgut, V.~R{\"o}dl, and M.~Schacht, \emph{Ramsey properties of random
  discrete structures}, Submitted.

\bibitem{GeSchSte:K5Free}
S.~Gerke, T.~Schickinger, and A.~Steger, \emph{{$K_5$}-free subgraphs of random
  graphs}, Random Structures Algorithms \textbf{24} (2004), no.~2, 194--232.

\bibitem{JaLuRu:Book}
S.~Janson, T.~{\L}uczak, and A.~Ruci{\'n}ski, \emph{Random graphs},
  Wiley-Interscience, New York, 2000.

\bibitem{KLR}
Y.~Kohayakawa, T.~{\L}uczak, and V.~R{\"o}dl, \emph{On {$K\sp 4$}-free
  subgraphs of random graphs}, Combinatorica \textbf{17} (1997), no.~2,
  173--213.

\bibitem{KoRoSch:TuranOld}
Y.~Kohayakawa, V.~R{\"o}dl, and M.~Schacht, \emph{The {T}ur\'an theorem for
  random graphs}, Combin. Probab. Comput. \textbf{13} (2004), no.~1, 61--91.

\bibitem{Mub10}
D.~Mubayi, \emph{Books versus triangles}, Preprint (arXiv:1002.1492).

\bibitem{RoRu:RamseyThreshold95}
V.~R{\"o}dl and A.~Ruci{\'n}ski, \emph{Threshold functions for {R}amsey
  properties}, J. Amer. Math. Soc. \textbf{8} (1995), no.~4, 917--942.

\bibitem{Sauer:VC}
N.~Sauer, \emph{On the density of families of sets}, J. Combinatorial Theory
  Ser. A \textbf{13} (1972), 145--147.

\bibitem{Schacht:KLR}
M.~Schacht, \emph{Extremal results for random discrete structures}, Submitted.

\bibitem{Shelah:VC}
S.~Shelah, \emph{A combinatorial problem; stability and order for models and
  theories in infinitary languages}, Pacific J. Math. \textbf{41} (1972),
  247--261.

\bibitem{S68}
M.~Simonovits, \emph{A method for solving extremal problems in graph theory,
  stability problems}, Theory of Graphs (Proc. Colloq., Tihany, 1966), Academic
  Press, New York, 1968, pp.~279--319.

\bibitem{Sperner}
E.~Sperner, \emph{{Ein Satz \"uber Untermengen einer endlichen Menge.}}, Math.
  Z. \textbf{27} (1928), 544--548 (German).

\bibitem{SudVuResil}
B.~Sudakov and V.~Vu, \emph{Local resilience of graphs}, Random Structures
  Algorithms \textbf{33} (2008), 409--433.

\bibitem{Tur}
P.~Tur{\'a}n, \emph{Eine {E}xtremalaufgabe aus der {G}raphentheorie}, Mat. Fiz.
  Lapok \textbf{48} (1941), 436--452.

\end{thebibliography}
\end{document}